\documentclass[a4paper,11pt]{article}

\usepackage[T1]{fontenc}
\usepackage{lmodern}

\usepackage{dsfont}

\usepackage[latin1]{inputenc}
\usepackage{amsmath}
\usepackage{amsthm}
\usepackage{amssymb}
\usepackage{mathrsfs}
\usepackage{graphicx}
\usepackage[all]{xy}
\usepackage{hyperref}

\usepackage{stmaryrd}
\usepackage{caption}

\usepackage{abstract} 

\newtheorem{thm}{Theorem}[section]
\newtheorem{cor}[thm]{Corollary}
\newtheorem{claim}[thm]{Claim}

\newtheorem{lemma}[thm]{Lemma}
\newtheorem{prop}[thm]{Proposition}

\theoremstyle{definition}
\newtheorem{definition}[thm]{Definition}
\newtheorem{ex}[thm]{Example}
\newtheorem{remark}[thm]{Remark}

\newtheorem{problem}[thm]{Problem}

\title{On the geometry of van Kampen diagrams of graph products of groups}
\date{\today}
\author{Anthony Genevois}

\begin{document}

\maketitle

\begin{abstract}
In this article, we propose a geometric framework dedicated to the study of van Kampen diagrams of graph products of groups. As an application, we find information on the word and the conjugacy problems. The main new result of the paper deals with the computation of conjugacy length functions. More precisely, if $\Gamma$ is a finite graph and $\mathcal{G}= \{ G_u \mid u \in V(\Gamma) \}$ a collection of finitely generated groups indexed by the vertices of $\Gamma$, then
$$\max\limits_{u \in V(\Gamma)} \mathrm{CLF}_{G_u}(n) \leq \mathrm{CLF}_{\Gamma \mathcal{G}}(n) \leq (D+1) \cdot n + \max\limits_{\Delta \subset \Gamma \ \text{complete}} \sum\limits_{u \in V(\Delta)} \mathrm{CLF}_{G_u}(n)$$
for every $n \geq 1$, where $D$ denotes the maximal diameter of a connected component of the opposite graph $\Gamma^{\mathrm{opp}}$. As a consequence, a graph product of groups with linear conjugacy length functions has linear conjugacy length function as well.
\end{abstract}

\tableofcontents

\section{Introduction}

Given a graph $\Gamma$ and a collection of groups $\mathcal{G}= \{ G_u \mid u \in V(\Gamma) \}$ indexed by the vertex-set $V(\Gamma)$ of $\Gamma$, we define the \emph{graph product} $\Gamma \mathcal{G}$ as the quotient of the free product of \emph{vertex-groups}
$$ \left( \underset{u \in V(\Gamma)}{\ast} G_u \right) / \langle \langle [g,h]=1 \ \text{if $g \in G_u$ and $h \in G_v$ for some $\{u,v\} \in E(\Gamma)$} \rangle \rangle$$
where $E(\Gamma)$ denotes the edge-set of $\Gamma$. The operation is usually thought of as an interpolation between free products and direct sums, as $\Gamma \mathcal{G}$ reduces to the free product of $\mathcal{G}$ if $\Gamma$ does not contain edges and to the direct sum of $\mathcal{G}$ if $\Gamma$ is a complete graph. Graph products of groups have been introduced by Elisabeth Green in her thesis \cite{GreenGP} as a common generalisation of two intensively studied families of groups, namely right-angled Artin groups (which correspond to graph products of infinite cyclic groups) and right-angled Coxeter groups (which correspond to graph products of cyclic groups of order two). 

Since Green's thesis, graph products of groups have been studied greatly, and powerful tools are now available to study them, including a normal form \cite[Theorem 3.9]{GreenGP}, canonical decompositions as amalgamated products \cite[Lemma 3.20]{GreenGP} (and their associated Bass-Serre trees \cite[Section 6]{MinasyanOsin}), a theory of parabolic subgroups \cite{TitsGP}, and different geometric models (namely, a right-angled building \cite{DavisBuildingsCAT}, a CAT(0) cube complex \cite{DavisBuildingsCAT} and a quasi-media graph \cite{Qm} (see \cite[Section 2]{AutGPMartin} for more details on the connections between all these models)). In this list, the normal form is really foundational, as all the other tools are based on it, making graph products combinatorial objects. In this paper, we would like to think of graph products as geometric objects instead. A first step in this direction has been realised in \cite{Qm}, proposing a quasi-median graph as the good geometric point of view. Although most of the tools mentioned above can be studied in a very nice way in this geometric framework, the normal form is necessary to show that the quasi-median geometry actually applies. 

The purpose of this article is to complete the geometric description of graph products of groups. More precisely, given a graph $\Gamma$ and a collection of groups $\mathcal{G}$ indexed by $V(\Gamma)$, we are interested in \emph{diagrams} over the presentation
$$\mathcal{Q}=\left\langle x_g, \ g \in \bigcup\limits_{u \in V(\Gamma)} G_u \backslash \{1\} \left| \begin{array}{l} [x_g,x_h]=1 \ \text{if $g \in G_u$ and $h \in G_v$ where $\{u,v\} \in E(\Gamma)$} \\ x_gx_h=x_{gh} \ \text{if $g,h,gh \in G_u \backslash \{1\}$ for some $u \in V(\Gamma)$} \end{array} \right. \right\rangle$$
of $\Gamma \mathcal{G}$, which are planar $2$-complexes whose $2$-cells are squares and triangles. (We refer to Section \ref{section:def} for precise definitions.) In order to study the geometry of these diagrams, we introduce \emph{dual curves}. More precisely, given a diagram $D$ over $\mathcal{Q}$, a {dual curve} $\alpha$ is a minimal subset satisfying the following three conditions:
\begin{itemize}
	\item for every edge $e \subset D$, the intersection $\alpha \cap D$ is either empty or the midpoint of $e$;
	\item for every square $C \subset D$, if $\alpha$ contains the midpoint of an edge of $C$ then it contains the straight line which links it to the midpoint of the opposite edge;
	\item for every triangle $T \subset D$, if $\alpha$ contains the midpoint of an edge of $T$ then it contains the three straight lines which link the center of $T$ to the midpoints of all its edges.
\end{itemize}
See Figure \ref{dualcurve} for an example. These dual curves generalise dual curves in disc diagrams of CAT(0) cube complexes \cite{MR1347406, BigWise} and are directly inspired by the definition of hyperplanes in quasi-median graphs \cite{Qm}. The main result we prove about dual curves in diagrams of graph products is the following:

\begin{prop}\label{intro:dualcurve}
Let $D$ be a diagram and $C \subset D$ a dual curve. If $C$ is not a circle, then an embedded circle in $C$ cannot be homotopically trivial in $D$. 
\end{prop}

As a consequence, dual curves in \emph{van Kampen diagrams} (ie., in simply connected diagrams) must circles or trees. This observation leads to a simple and geometric proof of the normal form in graph products. More precisely, 
given a graph $\Gamma$ and a collection of groups $\mathcal{G}$ indexed by $V(\Gamma)$, if $w=s_1 \cdots s_n$ is a word of length $n \geq 1$ written over $\bigcup\limits_{u \in V(\Gamma)} G_u$, then $s_i$'s are called the \emph{syllables} of $w$, and we say that $w$ is \emph{graphically reduced}
\begin{itemize}
	\item if $n=1$ and $s_1=1$, or $s_i \neq 1$ for every $1 \leq i \leq n$;
	\item and if, for every $1 \leq i < j \leq n$, either $s_i$ and $s_j$ do not belong to the same vertex-group or there exists some $i< k < j$ such that $s_k$ does not belong to a vertex-group adjacent to the common vertex-group containing $s_i$ and $s_j$. 
\end{itemize}
It is clear that any element of a graph product can be written as a graphically reduced word. Then:

\begin{thm}\label{intro:normalform}
Let $\Gamma$ be a simplicial graph and $\mathcal{G}$ a collection of groups indexed by $V(\Gamma)$. If $w_1$ and $w_2$ are two graphically reduced words which are equal in $\Gamma \mathcal{G}$, then it is possible to transform $w_1$ into $w_2$ by permuting successive syllables which belong to adjacent vertex-groups. 
\end{thm}

By studying more carefully van Kampen diagrams over $\mathcal{Q}$, we are also able to compute precisely the Dehn function of a graph product of finitely presented groups. More precisely:

\begin{thm}\label{intro:Dehn}
Let $\Gamma$ be a finite simplicial graph and $\mathcal{G}$ a collection of finitely presented groups indexed by $V(\Gamma)$. 
\begin{itemize}
	\item If $\Gamma$ is a clique, then $\delta_{\Gamma \mathcal{G}} \sim \max( \delta_G, G \in \mathcal{G})$.
	\item If $(\Gamma, \mathcal{G})$ satisfies Meier's condition and is not a clique, then $$\delta_{\Gamma \mathcal{G}} \sim \max \left( n \mapsto n, \widetilde{\delta_G}, G \in \mathcal{G} \right)= \max \left( \overline{\delta_G}, G \in \mathcal{G} \right).$$
	\item If $(\Gamma, \mathcal{G})$ does not satisfy Meier's condition and is not a clique, then $$\delta_{\Gamma \mathcal{G}} \sim \max \left( n \mapsto n^2, \widetilde{\delta_G}, G \in \mathcal{G} \right).$$
\end{itemize}
\end{thm}

We need to explain the terminology used in this statement. Given a graph $\Gamma$ and a collection of groups $\mathcal{G}$ indexed by $V(\Gamma)$, we say that $(\Gamma, \mathcal{G})$ satisfies \emph{Meier's condition} if 
\begin{itemize}
	\item $\Gamma$ is square-free;
	\item no two infinite vertex-groups are adjacent;
	\item and the link of a vertex labelled by an infinite vertex-group is complete.
\end{itemize}
Also, for every $G \in \mathcal{G}$, we denote by $\widetilde{\delta_G}$ the Dehn function $\delta_G$ of $G$ if $G$ is adjacent to all the other vertex-groups, and its \emph{negative closure} $\overline{\delta_G}$ otherwise, ie., 
$$\overline{\delta_G} : n \mapsto \max \left\{ \sum\limits_{i=1}^r \delta_G(n_i) \mid r \geq 1, \ \sum\limits_{i=1}^r n_i=n \right\}.$$
Previous works on Dehn functions of graph products include \cite{MeierDehn, CohenGP, AlonsoDehnGP}, where estimates are proved. Here, we identify precisely the (equivalence class of the) Dehn function. However, our argument is fundamentally based on \cite[Theorem~2.4]{MeierDehn} (see Proposition \ref{prop:Dehn} below), which we reprove by using van Kampen diagrams (while Meier uses general results proved by Brick about Dehn functions of amalgamated products).

Thus, so far we have shown how to solve the word problem in graph products by proving the normal form, and next we have proved a quantitative version in this spirit by computing Dehn functions. Now, we would like to do something similar for the conjugacy problem by considering \emph{annular diagrams} over $\mathcal{Q}$ (see Definition \ref{def:annulardiag} for a precise definition). First, given a graph $\Gamma$ and a collection of groups $\mathcal{G}$ indexed by $V(\Gamma)$, say that a word $s_1 \cdots s_n$ written over $\bigcup\limits_{u \in G_u} G_u$ is \emph{graphically cyclically reduced} if it is graphically reduced and if there does not exist $1 \leq i< j \leq n$ such that the vertex-group containing $s_i$ (resp. $s_j$) is adjacent to the vertex-group containing $s_k$ for every $1 \leq k < i$ (resp. for every $j<k \leq n$). It is not difficult to show that any element of a graph product is conjugate to a graphically cyclically reduced word (see Lemma \ref{lem:graphicallycyclicallyreduced}). We prove the following statement, which is well-known for right-angled Artin groups (see \cite[Lemma 9]{ConjRAAGlinear}) and which can also be found in \cite[Lemma 3.12]{FerovConjGP} (where it is proved combinatorially):

\begin{thm}\label{intro:conj}
Let $a,b \in \Gamma \mathcal{G}$ be two graphically cyclically reduced elements. Then $a$ and $b$ are conjugate if and only if there exist graphically reduced words $x_1 \cdots x_r p_1 \cdots p_n$ and $y_1 \cdots y_r q_1 \cdots q_n$ such that
\begin{itemize}
	\item $a=x_1 \cdots x_r p_1 \cdots p_n$ and $b=y_1 \cdots y_r q_1 \cdots q_n$;
	\item the $p_i$'s and the $q_i$'s are floating syllables of $a$ and $b$ respectively, and $p_i$ and $q_{i}$ are conjugate in a vertex-group;
	\item $y_1 \cdots y_s$ can be obtained from $x_1 \cdots x_r$ by permuting two consecutive syllables which belong to adjacent vertex-groups and by performing cyclic permutations. 
\end{itemize}
\end{thm}

This statement requires the following definition: given a graphically reduced word $w=s_1 \cdots s_n$, we say that $s_i$ is a \emph{floating syllable} if the vertex-group containing $s_i$ is adjacent to the vertex-group containing $s_j$ for every $j \neq i$.

A consequence of Theorem \ref{intro:conj} is that a graph product of groups with solvable conjugacy problems has a solvable conjugacy problem as well. Although this corollary has been already proved by Green \cite[Theorem 3.24]{GreenGP}, her proof is based on general results about the conjugacy problem in amalgamated products. Theorem \ref{intro:conj} leads to a simpler solution of the conjugacy problem in graph products. (See Example \ref{ex:conj} for an illustration of the method.) 

Next, we focus on \emph{conjugacy length functions}. Given a group $G$ endowed with a finite generating set $S$, the \emph{conjugacy length function} $\mathrm{CLF}_G$ gives, for every $n \geq 1$, the maximal length of a shortest conjugator between two conjugate elements $a,b \in G$ satisfying $\|a\|_S+ \|b\|_S \leq n$. Otherwise saying, for every $n \geq 1$, $\mathrm{CLF}_G(n)$ is the smallest number such that, for every $a,b \in G$ satisfying $\|a\|_S + \|b\|_S \leq n$, there exists some $c \in G$ such that $a=cbc^{-1}$ and $\|c\|_S \leq \mathrm{CLF}_G(n)$. Formally,
$$\mathrm{CLF}_G : n \mapsto \max\limits_{a,b \in G, \ \|a\|_S + \|b \|_S \leq n} \ \min\limits_{c \in G, \ cac^{-1}=b} \|c \|_S.$$
It is worth noticing that, up to equivalence (see Definition \ref{def:equivalentfunctions}), the conjugacy length function of a group does not depend on the finite generating set we choose. Conjugacy length functions are the analogues of Dehn functions for the conjugacy problem. For instance, it is not difficult to show that a finitely generated group with a solvable word problem has a solvable conjugacy problem if and only if its conjugacy length function is recursive. However, contrary to Dehn functions, conjugacy length functions do not provide quasi-isometric invariants. Indeed, in \cite{CPfiniteindex}, a finitely presented group with a solvable conjugacy problem but which contains a finite-index subgroup with an unsolvable conjugacy problem is constructed. 

So far, estimates on conjugacy length functions have been obtained from several classes of groups, included hyperbolic groups \cite[Lemma III.$\Gamma$.29]{MR1744486}, CAT(0) groups \cite[Theorem III.$\Gamma$.1.12]{MR1744486}, mapping class groups \cite{CLFmcg}, free solvable groups \cite{CLFwreathfreesolvable}, wreath products \cite{CLFwreathfreesolvable}, group extensions \cite{CLFextension}, right-angled Artin groups \cite{ConjRAAGlinear}, and very recently cocompact special groups \cite{CLFhierarchy}. Thanks to Theorem \ref{intro:conj}, we are able to add graph products to this list. More precisely, we prove:

\begin{thm}\label{intro:CLF}
Let $\Gamma$ be a finite simplicial graph and $\mathcal{G}$ a collection of finitely generated groups indexed by $V(\Gamma)$. Then
$$\max\limits_{u \in V(\Gamma)} \mathrm{CLF}_{G_u}(n) \leq \mathrm{CLF}_{\Gamma \mathcal{G}}(n) \leq (D+1) \cdot n + \max\limits_{\Delta \subset \Gamma \ \text{complete}} \sum\limits_{u \in V(\Delta)} \mathrm{CLF}_{G_u}(n)$$
for every $n \geq 1$, where $D$ denotes the maximal diameter of a connected component of the opposite graph $\Gamma^{\mathrm{opp}}$. 
\end{thm}

Given a graph $\Gamma$, the \emph{opposite graph} $\Gamma^{\mathrm{opp}}$ is the graph whose vertex-set is $V(\Gamma)$ and whose edges link two vertices if they are not adjacent in $\Gamma$. As a consequence of the previous theorem, a graph product of groups with linear conjugacy length functions has a linear conjugacy length function as well. This observation encompasses right-angled Artin groups and right-angled Coxeter groups (see Corollary \ref{cor:CLFraag} for a more precise estimate). 

\medskip
As a final remark, we have to mention that a simple proof of the normal form of graph products can also be found in \cite[Section 4]{HsuWiseGP}. Also, the notion of dual curves we use in this article have been introduced recently and independently in \cite{AcylHypGP}. In particular, Proposition \ref{intro:dualcurve} can be thought of as a generalisation of \cite[Lemma~2.3]{AcylHypGP}.

\paragraph{Some problems.} Let us conclude this introduction by mentioning a few interesting problems which could be solved thanks to the formalism introduced in this paper. The first of these problems is the following:

\begin{problem}
Describe centralisers in graph products of groups.
\end{problem}

A solution already appear in \cite{CentralisersGP}, but the proof is based on long combinatorial arguments related to normal forms. It would be interesting to have a simpler and geometric proof. Our second problem is:

\begin{problem}
For every $n \geq 2$, how to determine whether or not an element of a graph product is an $n$th power?
\end{problem}

The expected answer is the following. Up to conjugacy, we may suppose that the word $w$ we are looking at is graphically cyclically reduced. Write such a word as $x_1 \cdots x_r p_1 \cdots p_s$ where $p_1, \ldots, p_s$ are all the floating syllables of $w$. Then $w$ is an $n$th power if and only if each $p_i$ is an $n$th power in the vertex-group which contains it and $x_1 \cdots x_r$ can be obtained from a word of the form $(x_{i_1} \cdots x_{i_k})^n$ by permuting consecutive syllables if they belong to adjacent vertex-groups. 

A solution to this problem would be the starting point towards a solution of the third and last problem we propose:

\begin{problem}
Let $\Gamma$ be a graph and $\mathcal{G}$ a collection of groups indexed by $V(\Gamma)$. Suppose that $\Gamma\mathcal{G}$ contains a subgroup $H$ isomorphic to $\langle a,b,c \mid a^m=b^nc^p \rangle$ for some $m,n,p \geq 2$. Does there exist a complete subgraph $\Lambda \subset \Gamma$ such that $H$ is included into the subgroup $\langle \Lambda \rangle$ generated by the vertex-groups labelling the vertices of $\Lambda$?
\end{problem}

\noindent
A positive answer can be found in \cite[Corollary 1.7]{TitsGP} for right-angled Artin groups.

\paragraph{Organisation of the paper.} In Section \ref{section:def}, we begin by giving all the definitions about graph products of groups and diagrams over group presentations which will be used in the rest of the article. Next, we prove Proposition \ref{intro:dualcurve} about dual curves and Theorem \ref{intro:normalform} about the normal form of graph products in Section \ref{section:normalform}. We also discuss a few consequences related to the word problem. Section \ref{section:Dehn} is dedicated to the proof of Theorem \ref{intro:Dehn}. In Section \ref{section:conj}, we prove Theorem \ref{intro:conj} and we discuss a few consequences related to the conjugacy problem. Finally, Section \ref{section:CLF} is dedicated to conjugacy length functions. There, we prove Theorem \ref{intro:CLF} and we show that it is possible to get a linear control on the sizes of shortest conjugators between specific elements of graph products without any assumption on vertex-groups.

\section{Graph products and diagrams}\label{section:def}

\noindent
In this section, we specify the definitions and notations related to graph products and diagrams which we will be used in the rest of the article. We begin by defining graph products of groups.

\paragraph{Graph products of groups.} Let $\Gamma$ be a simplicial graph and $\mathcal{G}= \{ G_u \mid u \in V(\Gamma)\}$ a collection of groups, called \emph{vertex-groups}, indexed by the vertex-set $V(\Gamma)$ of $\Gamma$. The \emph{graph product} $\Gamma \mathcal{G}$ is defined by the quotient
$$ \left( \underset{u \in V(\Gamma)}{\ast} G_u \right) / \langle \langle [g,h]=1 \ \text{if $g \in G_u$ and $h \in G_v$ for some $\{u,v\} \in E(\Gamma)$} \rangle \rangle$$
of the free product of vertex-groups, where $E(\Gamma)$ denotes the edge-set of $\Gamma$. 

\medskip \noindent
For instance, if $\Gamma$ has no edges then $\Gamma \mathcal{G}$ is the free product of $\mathcal{G}$; and if $\Gamma$ is a complete graph then $\Gamma \mathcal{G}$ is the direct sum of $\mathcal{G}$. 

\medskip \noindent
\textbf{Convention:} In all the article, vertex-groups are always assumed to be non-trivial. 

\medskip \noindent
Notice that this assumption is not restrictive. Indeed, if $\Phi$ denotes the subgraph of $\Gamma$ generated by the vertices which are labelled by non-trivial vertex-groups and if $\mathcal{H}$ denotes the collection of groups obtained from $\mathcal{G}$ by removing the trivial vertex-groups, then $\Gamma \mathcal{G}$ is isomorphic to the new graph product $\Phi \mathcal{H}$.

\paragraph{Diagrams.} Now, we want to consider diagrams over group presentations. The reference we essentially follow is \cite{LS}. When reading the definitions below, we refer to Figure \ref{diag} for examples; there, $\partial \Delta_2$ is labelled by $(a^{-1}a^{2}a^{-1}, a^3a^{-3})$, and $\partial \Delta_3$ by $b^{-1}a^{-1}ba^4b^{-1}a^{-2}ba^{-2}$. 
\begin{figure}
\begin{center}
\includegraphics[trim={0 6.5cm 10cm 0},clip,scale=0.34]{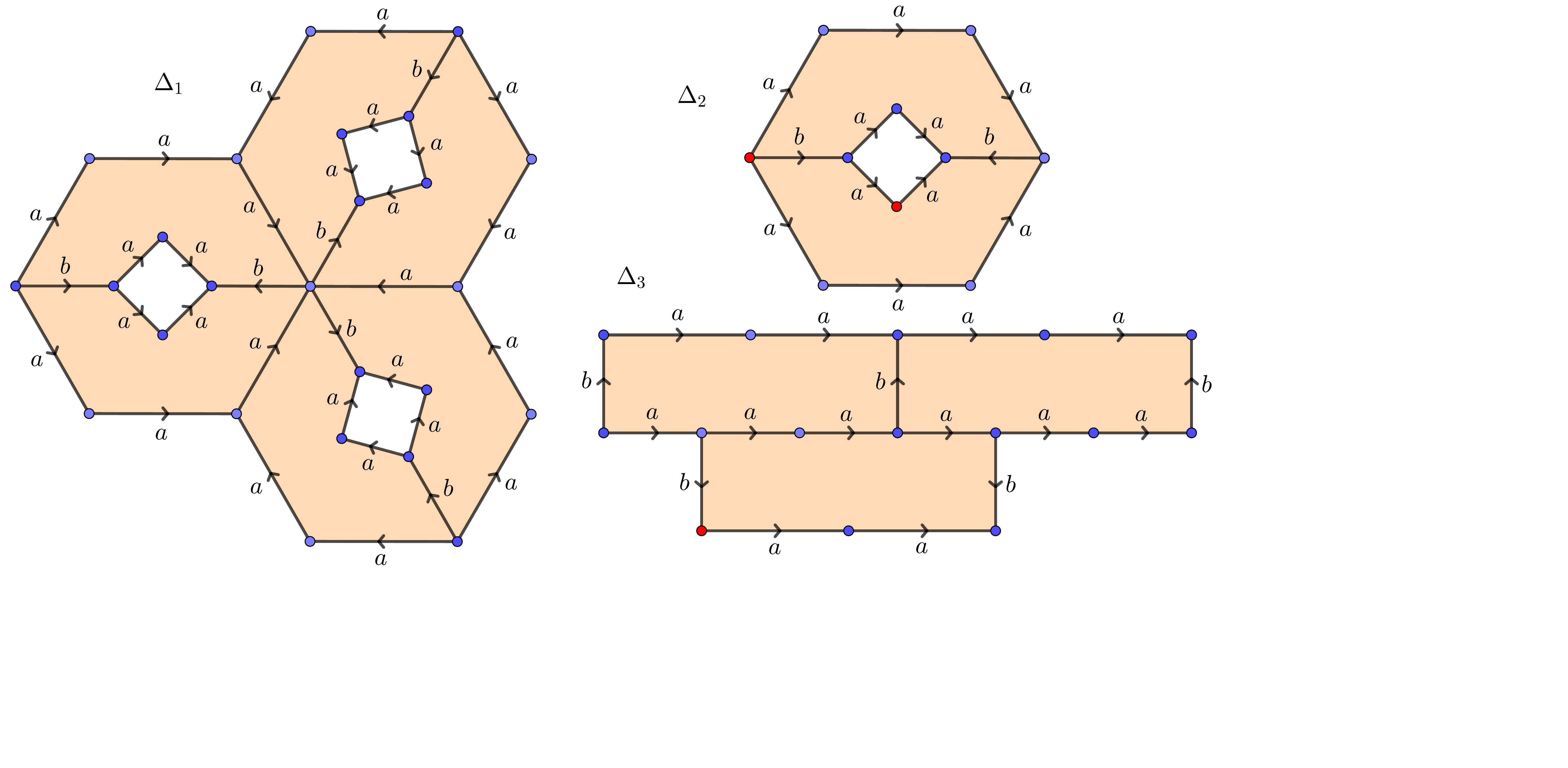}
\caption{A general diagram, a van Kampen diagram, and an annular diagram over the presentation $\langle a,b \mid ba^2b^{-1}=a^3 \rangle$.}
\label{diag}
\end{center}
\end{figure}

\begin{definition}
Let $\mathcal{P}= \langle X \mid R \rangle$ be a group presentation, and $D$ a finite $2$-complex embedded into the plane whose edges are oriented and labelled by elements of $X$. An oriented path $\gamma$ in the one-skeleton of $D$ can be written as a concatenation $e_1^{\epsilon_1} \cdots e_n^{\epsilon_n}$, where $\epsilon_1, \ldots, \epsilon_n \in \{+1,-1\}$ and where each $e_i$ is an edge endowed with the orientation coming from $D$. Then the word \emph{labelling} $\gamma$ is $\ell_1^{\epsilon_1} \cdots \ell_n^{\epsilon_n}$ where $\ell_i$ is the label of $e_i$ for every $1 \leq i \leq n$. If, for every $2$-cell $F$ of $D$, the word labelling the boundary of $F$ (an arbitrary basepoint and an arbitrary orientation being fixed) is a cyclic permutation of a relation of $R$ or of the inverse of a relation of $R$, then $D$ is a \emph{diagram over $\mathcal{P}$}.
\end{definition}

\noindent
We will be interested in two specific types of diagrams, namely van Kampen diagrams and annular diagrams, which are respectively related to the word and the conjugacy problems.

\begin{definition}
Let $\mathcal{P}= \langle X \mid R \rangle$ be a group presentation. A \emph{van Kampen diagram over $\mathcal{P}$} is a simply connected diagram over $\mathcal{P}$ with a fixed vertex in its boundary (ie., the intersection between $D$ and the closure of $\mathbb{R}^2 \backslash D$). A \emph{boundary cycle} of $D$ is a cycle $\alpha$ of minimal length which contains all the edges in the boundary of $D$ which does not cross itself, in the sense that, if $e$ and $e'$ are consecutive edges of $\alpha$ with $e$ ending at a vertex $v$, then $e^{-1}$ and $e'$ are adjacent in the cyclically ordered set of all edges of $D$ beginning at $v$. The \emph{label} of the boundary of $D$ is the word labelling the boundary cycle of $D$ which begins at the basepoint of $D$ and which turns around $D$ clockwise. 
\end{definition}

\noindent
The connection between van Kampen diagrams and the word problem is made explicit by the following statement. We refer to \cite[Theorem V.1.1 and Lemma V.1.2]{LS} for a proof.

\begin{prop}
Let $\mathcal{P}= \langle X \mid R \rangle$ be a presentation of a group $G$ and $w \in X^{\pm}$ a non-empty word. There exists a van Kampen diagram over $\mathcal{P}$ whose boundary is labelled by $w$ if and only if $w=1$ in $G$.
\end{prop}

\noindent
Next, let us consider annular diagrams.

\begin{definition}\label{def:annulardiag}
Let $\mathcal{P}= \langle X \mid R \rangle$ be a group presentation. An \emph{annular diagrams over $\mathcal{P}$} is a diagram $D$ over $\mathcal{P}$ such that $\mathbb{R}^2 \backslash D$ has exactly two connected components, endowed with a fixed vertex in each connected component of its boundary (ie., the  intersection between $D$ and the closure of $\mathbb{R}^2 \backslash D$). The \emph{inner boundary} (resp. \emph{outer boundary}) of $D$, denoted by $\partial_\text{inn}D$ (resp. $\partial_\text{out}D$), is the intersection of $D$ with the bounded (resp. unbounded) component of $\mathbb{R}^2 \backslash D$. A cycle of minimal length (that does not cross itself) which contains all the edges in the outer (resp. inner) boundary of $D$ is an \emph{outer} (resp. \emph{inner}) \emph{boundary cycle} of $D$. The \emph{label} of the boundary of $D$ is the couple $(w_1,w_2)$ where $w_1$ (resp. $w_2$) is the word labelling the inner (resp. outer) boundary cycle of $D$ which begins at the basepoint of $D$ and which turns clockwise.
\end{definition}

\noindent
The connection between annular diagrams and the conjugacy problem is made explicit by the following statement. We refer to \cite[Lemmas V.5.1 and V.5.2]{LS} for a proof.

\begin{prop}
Let $\mathcal{P}= \langle X \mid R \rangle$ be a presentation of a group $G$ and $w_1,w_2 \in X^{\pm}$ two non-empty words. There exists an annular diagrams over $\mathcal{P}$ whose boundary is labelled by $(w_1,w_2)$ if and only if $w_1$ and $w_2$ are conjugate in $G$.
\end{prop}

\noindent
Now, we want to focus on diagrams over presentations of graph products of groups.

\paragraph{Diagrams of graph products.} Fix a simplicial graph $\Gamma$ and a collection of groups $\mathcal{G}$ indexed by $V(\Gamma)$. We will be interested in diagrams over the following presentation of the graph product $\Gamma \mathcal{G}$: 
$$\mathcal{Q}= \left\langle x_g, \ g \in \bigcup\limits_{u \in V(\Gamma)} G_u \backslash \{1\} \left| \begin{array}{l} [x_g,x_h]=1 \ \text{if $g \in G_u$ and $h \in G_v$ where $\{u,v\} \in E(\Gamma)$} \\ x_gx_h=x_{gh} \ \text{if $g,h,gh \in G_u \backslash \{1\}$ for some $u \in V(\Gamma)$} \end{array} \right. \right\rangle$$
We emphasize that the neutral element is not a generator in $\mathcal{Q}$. Consequently, an edge of a diagram over $\mathcal{Q}$ cannot be labelled by a trivial element of $\Gamma \mathcal{G}$. Also, notice that $2$-cells of diagrams over $\mathcal{Q}$ are squares and triangles.

\medskip \noindent
The fundamental tool we will use to study diagrams of graph products is the notion of dual curves. Notice that the definition below only depends on the underlying $2$-complex of a diagram (forgetting the orientations of the edges and their labels), so it may be convenient to think of diagrams as $2$-complexes when dealing with dual curves. 

\begin{definition}
Let $D$ be a diagram over $\mathcal{Q}$. A \emph{dual curve} of $D$ is a minimal subset $\alpha \subset D$ satisfying the following three conditions:
\begin{itemize}
	\item for every edge $e \subset D$, the intersection $\alpha \cap D$ is either empty or the midpoint of $e$;
	\item for every square $C \subset D$, if $\alpha$ contains the midpoint of an edge of $C$ then it contains the straight line which links it to the midpoint of the opposite edge;
	\item for every triangle $T \subset D$, if $\alpha$ contains the midpoint of an edge of $T$ then it contains the three straight lines which link the center of $T$ to the midpoints of all its edges.
\end{itemize}
A \emph{singularity} of a dual curve is the center of a triangle which belongs to it. A dual curve is \emph{regular} if it does not contain any singularity. An \emph{arc} of a dual curve is a connected component of the complement of the singularities. Two dual curves are \emph{transverse} is their intersection contains the center of a square.
\end{definition}

\noindent
In Figure \ref{dualcurve}, the orange dual curve is regular; the orange, green and red dual curves are pairwise transverse; and the blue and orange dual curves are not transverse. 
\begin{figure}
\begin{center}
\includegraphics[trim={0 14cm 41cm 0},clip,scale=0.4]{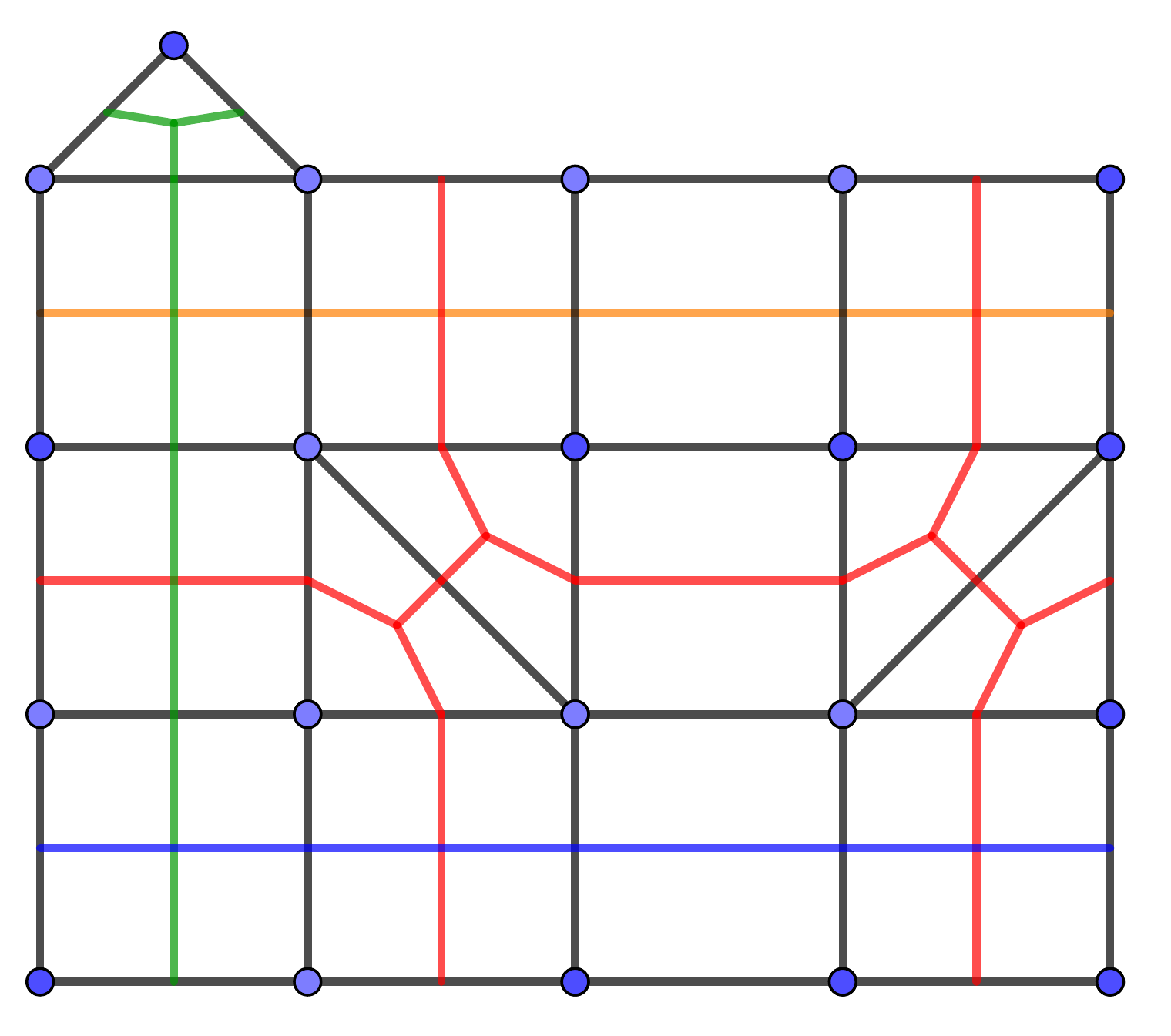}
\caption{Some dual curves of a diagram.}
\label{dualcurve}
\end{center}
\end{figure}

\medskip \noindent
Now, let us collect a few elementary observations related to dual curves in diagrams. Before stating our next lemma, notice that the edges of diagrams over $\mathcal{Q}$ are naturally labelled by vertices of $\Gamma$, since they are labelled by elements of vertex-groups.

\begin{lemma}\label{lem:labelhyp}
Let $\Gamma$ be a simplicial graph, $\mathcal{G}$ a collection of groups indexed by $V(\Gamma)$, and $D$ diagram over $\mathcal{Q}$. If two edges of $D$ are crossed by the same dual curve, then they are labelled by the same vertex of $\Gamma$.
\end{lemma}

\begin{proof}
If $e$ and $f$ are two edges of $D$ crossed by the same dual curve, then there exists a sequence of edges $e_0=e, e_1,\ldots, e_{n-1}, e_n=f$ such that, for every $0 \leq i \leq n-1$, the edges $e_i$ and $e_{i+1}$ either belong to a common triangle or are parallel in a common square. Therefore, in order to prove our lemma, it is sufficient to notice that two edges which belong to a common triangle or which are parallel in a common square are labelled by the same vertex of $\Gamma$. But this observation is clear from the definition of the relations of $\mathcal{Q}$. 
\end{proof}

\noindent
Lemma \ref{lem:labelhyp} above justifies the following definition:

\begin{definition}
Let $\Gamma$ be a simplicial graph, $\mathcal{G}$ a collection of groups indexed by $V(\Gamma)$, and $D$ a diagram over $\mathcal{Q}$. The \emph{label} of a dual curve is the vertex of $\Gamma$ labelling all the edges it crosses. 
\end{definition}

\noindent
Another elementary observation on dual curves of diagrams is the following:

\begin{lemma}\label{lem:labeltransverse}
Let $\Gamma$ be a simplicial graph, $\mathcal{G}$ a collection of groups indexed by $V(\Gamma)$, and $D$ a diagram over $\mathcal{Q}$. If two dual curves of $D$ are transverse, then they are labelled by two adjacent vertices of $\Gamma$.
\end{lemma}

\begin{proof}
It is sufficient to notice that any two adjacent edges of a square of $D$ are labelled by adjacent vertices of $\Gamma$. This observation is clear from the definition of the relations of $\mathcal{Q}$. 
\end{proof}

\noindent
As an immediate consequence:

\begin{cor}
A dual curve cannot self-intersect.
\end{cor}\qed

\noindent
As a consequence of this corollary, an arc of a dual curve is homeomorphic to a segment, which implies that our next definition makes sense.

\begin{definition}
Let $\Gamma$ be a simplicial graph, $\mathcal{G}$ a collection of groups indexed by $V(\Gamma)$, $D$ a diagram over $\mathcal{Q}$, and $\alpha$ an arc of a dual curve. Two oriented edges $e$ and $e'$ \emph{have the same orientation along $\alpha$} if they are crossed by $\alpha$ and if the translation along $\alpha$ sends $e$ onto $e'$. 
\end{definition}

\noindent
Our last preliminary result on dual curves of diagrams is the following:

\begin{lemma}\label{lem:labelarc}
Let $\Gamma$ be a simplicial graph, $\mathcal{G}$ a collection of groups indexed by $V(\Gamma)$, $D$ a diagram over $\mathcal{Q}$, and $\alpha$ an arc of a dual curve. Any two oriented edges having the same orientation along $\alpha$ are labelled by the same generator of $\mathcal{Q}$. 
\end{lemma}

\begin{proof}
If $e$ and $f$ are two oriented edges of $D$ which have the same orientation along $\alpha$, then there exists a sequence of oriented edges $e_0=e, e_1,\ldots, e_{n-1}, e_n=f$ such that, for every $0 \leq i \leq n-1$, the edges $e_i$ and $e_{i+1}$ are parallel in a common square (and point towards the same direction). Therefore, in order to prove our lemma, it is sufficient to notice that two oriented edges which are parallel in a common square are labelled by the same generator of $\mathcal{Q}$. But this observation is clear from the definition of the relations of $\mathcal{Q}$. 
\end{proof}

\noindent
Of course, a van Kampen or annular diagram with a fixed boundary is far from being unique. Our goal now is to define a few elementary transformations which allow us to modify diagrams without modifying their boundaries. First, an \emph{inversion} is the operation which amounts to invert an oriented edge and to replace its label by its inverse. Next, \emph{hexagonal moves}, \emph{pentagonal moves}, \emph{flips} and \emph{square-reductions} are illustrated by Figure \ref{elementary}. Any one of these five operations will be referred to as an \emph{elementary move}. An elementary but fundamental observation is:
\begin{figure}
\begin{center}
\includegraphics[trim={0 13cm 5cm 0},clip,scale=0.42]{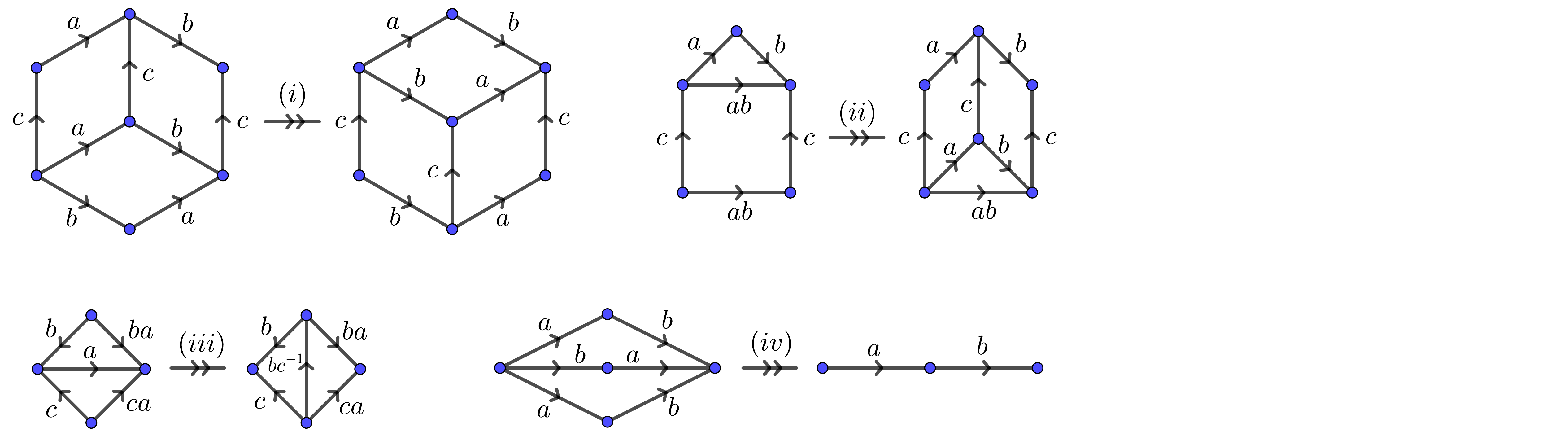}
\caption{A hexagonal move $(i)$, a pentagonal move $(ii)$, a flip $(iii)$, and a square-reduction $(iv)$.}
\label{elementary}
\end{center}
\end{figure}

\begin{lemma}
Let $\Gamma$ be a simplicial graph, $\mathcal{G}$ a collection of groups indexed by $V(\Gamma)$, and $D$ a diagram over $\mathcal{Q}$. If $D'$ is obtained from $D$ by an elementary move, then $D'$ is also a diagram over $\mathcal{Q}$. Moreover, if $D$ is a van Kampen diagram (resp. an annular diagram) then $D'$ is also a van Kampen diagram (resp. an annular diagram) and its boundary have the same label. 
\end{lemma}\qed

\section{Normal form and word problem}\label{section:normalform}

\noindent
In this section, our goal is to exploit the notion of dual curves introduced in the previous section in order to prove the normal form of graph products (Theorem \ref{thm:minlength} and Corollary~\ref{cor:graphicallyreduced} below) and to show how to solve the word problem (Corollary \ref{cor:WP}). We begin by defining the notion of reduced words which is relevant in the context of graph products of groups.

\begin{definition}
Let $\Gamma$ be a simplicial graph, $\mathcal{G}$ a collection of groups indexed by $V(\Gamma)$, and $w=s_1 \cdots s_n$ a word of length $n \geq 1$ written over $\bigcup\limits_{G \in \mathcal{G}} G$. The $s_i$'s will be referred to as the \emph{syllables} of $w$. The word $w$ is \emph{graphically reduced}
\begin{itemize}
	\item if $n=1$ and $s_1=1$, or $s_i \neq 1$ for every $1 \leq i \leq n$;
	\item and if, for every $1 \leq i < j \leq n$, either $s_i$ and $s_j$ do not belong to the same vertex-group or there exists some $i< k < j$ such that $s_k$ does not belong to a vertex-group adjacent to the common vertex-group containing $s_i$ and $s_j$. 
\end{itemize}
\end{definition}

\noindent
Notice that if $E(\Gamma)= \emptyset$, then $\Gamma \mathcal{G}$ is a free product and graphically reduced words coincide with alternating words. 

\medskip \noindent
The main result of this section is the following statement:

\begin{thm}\label{thm:minlength}
Let $\Gamma$ be a simplicial graph and $\mathcal{G}$ a collection of groups indexed by $V(\Gamma)$. A word written over $\bigcup\limits_{G \in \mathcal{G}} G \backslash \{ 1 \}$ has minimal length with respect to the generating set $\bigcup\limits_{G \in \mathcal{G}} G$ if and only if it is graphically reduced.
\end{thm}

\noindent
We begin by proving the following key proposition, which will be also useful when dealing with the conjugacy problem:

\begin{prop}\label{prop:circleindiag}
Let $D$ be a diagram and $C \subset D$ a dual curve. If $C$ is not a circle, then an embedded circle in $C$ cannot be homotopically trivial in $D$. 
\end{prop}

\begin{proof}
Suppose for contradiction that $C$ is not a circle and that it contains an embedded circle which is homotopically non-trivial in $D$. If $c \subset C$ is such a circle, define its area as the number of two-cells which are included into the inner component of $D \backslash c$. Choose an embedded homotopically non-trivial circle $c \in C$ which has minimal area. 

\medskip \noindent
Because $C$ is not a circle, $c$ has to contain at least one singularity of $C$, say $p \in C$. Let $\alpha \subset C$ denote the arc of $C$ whose closure contains $p$ and which is disjoint from $c$. Notice that, as a consequence of the minimality of $c$, $p$ cannot be included into the inner component of $D \backslash c$. Otherwise saying, if $\Delta$ denotes the triangle having $p$ as its center, then the edge of $\Delta$ which is disjoint from $c$ must be contained into the outer component of $D \backslash c$. So the union of the two-cells intersecting $c$ has the form given by Figure \ref{cycleshift}(a).

\medskip \noindent
Notice that, by applying pentagonal moves to $D$, it is possible to shift the triangles along the cycle. More precisely, by applying pentagonal moves to $D$, we create a new diagram $D'$ containing a subdiagram of the form given by Figure \ref{cycleshift}(b). Next, by applying flips to $D'$, we can remove all but one triangle from our cycle. More precisely, by applying flips to $D'$, we create a new diagram $D''$ containing a subdiagram of the form given by Figure \ref{cycleshift}(c). But such a diagram cannot exist. More precisely: 

\begin{claim}\label{claim:singlesing}
An embedded circle in a dual curve cannot contain exactly one singularity.
\end{claim}

\noindent
If a dual curve contains an embedded circle on which there is a single singularity, it precisely means that a dual curve contains an arc whose closure is a circle. It follows from Lemma \ref{lem:labelarc} that the edge of the triangle containing the singularity which is disjoint from the arc must be labelled by the neutral element, which is impossible because it is not a generator of $\mathcal{Q}$.
\begin{figure}
\begin{center}
\includegraphics[trim={0 19cm 15cm 0},clip,scale=0.43]{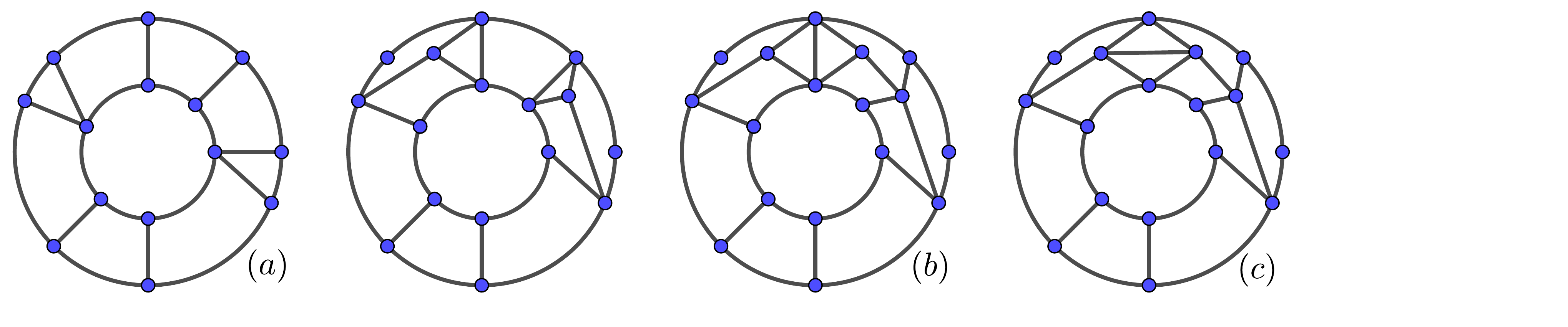}
\caption{Subdiagrams in the proof of Proposition \ref{prop:circleindiag}.}
\label{cycleshift}
\end{center}
\end{figure}
\end{proof}

\noindent
As van Kampen diagrams are simply connected by definition, an immediate consequence of Proposition \ref{prop:circleindiag} is the following statement:

\begin{cor}\label{cor:dualtree}
Let $\Gamma$ be a simplicial graph and $\mathcal{G}$ a collection of groups indexed by $V(\Gamma)$. In a van Kampen diagram over $\mathcal{Q}$, dual curves are circles or trees. 
\end{cor}\qed

\noindent
We are now ready to prove Theorem \ref{thm:minlength}.

\begin{proof}[Proof of Theorem \ref{thm:minlength}.]
Let $w=s_1 \cdots s_n$ be a word of length $n \geq 1$ written over $\bigcup\limits_{G \in \mathcal{G}} G$. 

\medskip \noindent
Assume that $w$ is not graphically reduced. Two cases may happen. First, if $n \geq 2$ and if $s_i=1$ for some $1 \leq i \leq n$, then $w$ can be shortened as 
$$s_1 \cdots s_{i-1}s_{i+1} \cdots s_n.$$ 
Second, if there exist $1 \leq i < j \leq n$ such that $s_i$ and $s_j$ belong to a common vertex-group $G \in \mathcal{G}$ and such that $s_k$ belong to a vertex-group adjacent to $G$ for every $i < k < j$, then $w$ can be shortened as
$$s_1 \cdots s_{i-1}(s_is_j) s_{i+1} \cdots s_{j-1}s_{j+1} \cdots s_n.$$
Thus, we have proved that a word which is not graphically reduced does not have minimal length, or equivalently, that a word of minimal length must be graphically reduced. 

\medskip \noindent
Now, let us assume that $w$ is graphically reduced. Fix a word $u$ of minimal length representing $w$ in $\Gamma \mathcal{G}$. As a consequence, there exists a van Kampen diagram $D$ whose boundary $\partial D$ is labelled by $uw^{-1}$. Notice that the word $u$ is also graphically reduced according to the previous paragraph. Therefore, it follows from our claim below that any dual curve of $D$ is regular, and that it crosses exactly two edges of $\partial D$, corresponding to two syllables of $u$ and $w^{-1}$ respectively. Consequently, the map which associates to a syllable $s$ of $u$ the syllable $r$ of $w^{-1}$ such that $r$ and $s$ labels two edges of $\partial D$ dual to the same dual curve defines a bijection from the set of syllables of $u$ to the set of syllables of $w^{-1}$. Therefore, the words $u$ and $w$ have the same length, which implies that $w$ must have minimal length.

\begin{claim}\label{claim:graphicallyreduceddiagram}
Let $D$ be a van Kampen diagram. If $\sigma \subset \partial D$ is a subsegment labelled by a graphically reduced word, then any dual curve of $D$ intersects $\sigma$ at most once.
\end{claim}

\noindent
Our argument is illustrated by Figure \ref{hyptwice}. Let $\sigma \subset \partial D$ be a subsegment, and let $r_1 \cdots r_m$ denote the word labelling it. Assume that there exists a dual curve of $D$ intersecting $\sigma$ at least twice, and let $e,f \subset \sigma$ be two edges crossed by the same dual curve $\gamma$ such that the subsegment $\xi \subset \sigma$ between them has minimal length. Let $1 \leq i < j \leq n$ denote the indices such that $r_i$ and $r_j$ label $e$ and $f$. Notice that the subsegment $\xi \subset \sigma$ is labelled by $r_{i+1} \cdots r_{j-1}$, and that $r_i$ and $r_j$ belong to a common vertex-group, say $G \in \mathcal{G}$, according to Lemma \ref{lem:labelhyp}. Because any dual curve of $D$ which intersects $\partial D$ has to be a tree according to Corollary \ref{cor:dualtree} and because no dual curve of $D$ intersects $\xi$ twice by minimality, it follows the dual curves of $D$ intersecting $\xi$ have to be transverse to $\gamma$. We deduce from Lemma \ref{lem:labeltransverse} that $r_k$ belongs to a vertex-group which is adjacent to $G$ for every $i <k <j$. Thus, we have proved that word $r_1 \cdots r_m$ labelling $\sigma$ is not graphically reduced. 
\begin{figure}
\begin{center}
\includegraphics[trim={0 16cm 39cm 0},clip,scale=0.5]{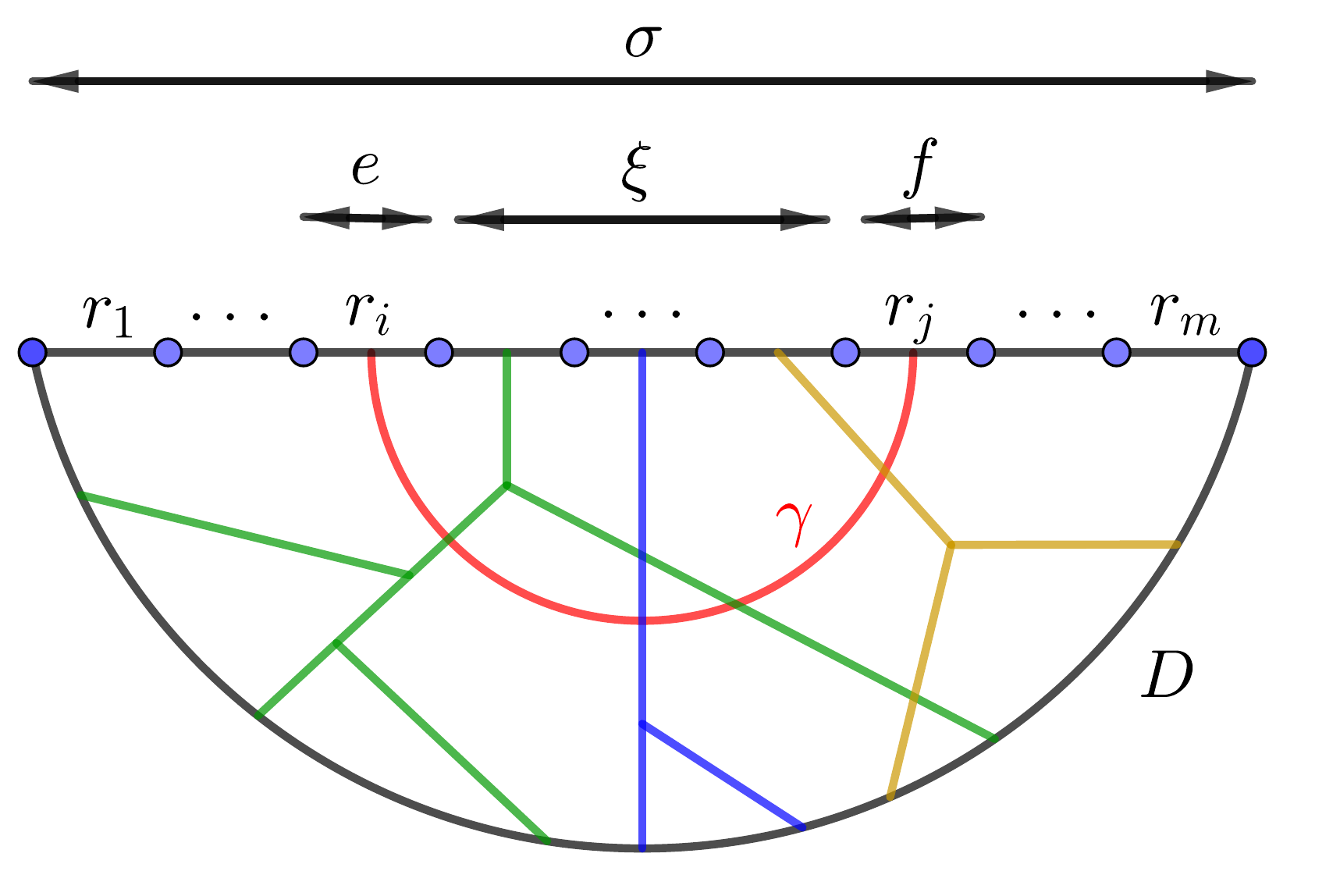}
\caption{Configuration from the proof of Claim \ref{claim:graphicallyreduceddiagram}.}
\label{hyptwice}
\end{center}
\end{figure}
\end{proof}

\noindent
Thanks to the arguments used above, we are now ready to show how to solve the word problem in graph products of groups.

\begin{cor}\label{cor:WP}
A graph product of groups admitting a solvable word problem has a solvable word problem as well. 
\end{cor}

\noindent
Our corollary will be essentially deduce from the following easy consequence of Theorem~\ref{thm:minlength}:

\begin{lemma}\label{lem:WP}
Let $\Gamma$ be a simplicial graph and $\mathcal{G}$ a collection of groups indexed by $V(\Gamma)$. If $w=s_1 \cdots s_n$ is a nonempty word of length $n \geq 2$ written over $\bigcup\limits_{G \in \mathcal{G}} G$ which is equal to $1$ in $\Gamma \mathcal{G}$ then $w$ is not graphically reduced.
\end{lemma}

\begin{proof}
If $s_i=1$ for some $1 \leq i \leq n$ then $w$ is not graphically reduced, so let us assume that $s_i \neq 1$ for every $1 \leq i  \leq n$. Clearly, because $w=1$ in $\Gamma \mathcal{G}$, the word $w$ does not have minimal length. It follows from Theorem \ref{thm:minlength} that $w$ is not graphically reduced. 
\end{proof}

\begin{proof}[Proof of Corollary \ref{cor:WP}.]
For every $u \in V(\Gamma)$, fix a presentation $\langle X_u \mid R_u \rangle$ of $G_u$ with a solvable word problem. Let $w$ be a word written over $\bigcup\limits_{u \in V(\Gamma)} X_u$. We think of $w$ as a word written over $\bigcup\limits_{u \in V(\Gamma)} G_u$. Iterate the following operations:
\begin{itemize}
	\item Use the algorithms available in vertex-groups to check if some syllables of $w$ are trivial. If so, remove them.
	\item If $w$ contains a subword of the form $a x_1 \cdots x_n b$ where $a$ and $b$ belongs to a common vertex-group which is adjacent to the vertex-groups containing the $x_i$'s, then replace $a x_1 \cdots x_n b_n$ with $(ab) x_1 \cdots x_n$.
\end{itemize}
Notice that the length (as a word over $\bigcup\limits_{u \in V(\Gamma)} G_u$) decreases if one of these two operations applies. As a consequence, after at most $\|w\|$ steps, where $\|w\|$ denotes the length of $w$, the process stops. If we get an empty word, then $w$ is trivial in $\Gamma \mathcal{G}$; and otherwise, the word we obtain must be graphically reduced, so that it follows from Lemma \ref{lem:WP} that $w$ represents a non-trivial element of $\Gamma \mathcal{G}$. 
\end{proof}

\noindent
In fact, in the proof of corollary \ref{cor:WP}, we have showed how to determine whether or not a word represents the trivial element of a graph product. This problem is equivalent to the word problem, but it is also possible to determine directly whether or not two words represent the same element in a graph product thanks to the following statement:

\begin{cor}\label{cor:graphicallyreduced}
Let $\Gamma$ be a simplicial graph and $\mathcal{G}$ a collection of groups indexed by $V(\Gamma)$. If $w_1$ and $w_2$ are two graphically reduced words which are equal in $\Gamma \mathcal{G}$, then it is possible to transform $w_1$ into $w_2$ by permuting successive syllables which belong to adjacent vertex-groups. 
\end{cor}

\begin{proof}
Since $w_1w_2^{-1}$ equals $1$ in $\Gamma \mathcal{G}$, there exists a van Kampen diagram $D$ whose boundary is labelled by $w_1w_2^{-1}$. Let $\sigma_1$ and $\sigma_2$ denote the two subsegments of $\partial D$ labelled by $w_1$ and $w_2^{-1}$ respectively. As a consequence of Claim \ref{claim:graphicallyreduceddiagram}, a dual curve of $D$ cannot intersect twice $\sigma_1$ or $\sigma_2$, so dual curves of $D$ are all regular. Let $\alpha_1, \ldots, \alpha_k$ be segments in $D$ between the common endpoints of $\sigma_1$ and $\sigma_2$ such that
\begin{itemize}
	\item each $\alpha_i$ intersects the dual curves of $D$ transversely;
	\item the subspace delimited by $\alpha_i$ and $\alpha_{i+1}$ contains exactly one intersection point between two dual curves.
\end{itemize}
See Figure \ref{word}. For convenience, set $\alpha_0= \sigma_1$ and $\alpha_{k+1}=\sigma_2$. For every $0 \leq i \leq k+1$, let $m_i$ denote the word obtained by reading the labels of the dual curves intersecting $\alpha_i$. In particular, $m_0=w_1$ and $m_{k+1}=w_2$. For every $0 \leq i \leq k$, the word $m_{i+1}$ is obtained from $m_i$ by permuting two consecutive letters corresponding to two transverse dual curves. It follows from Lemma \ref{lem:labeltransverse} that $m_{i+1}$ is obtained from $m_i$ by permuting two syllables which belong to adjacent vertex-groups.
\begin{figure}
\begin{center}
\includegraphics[trim={0 19cm 38cm 0},clip,scale=0.5]{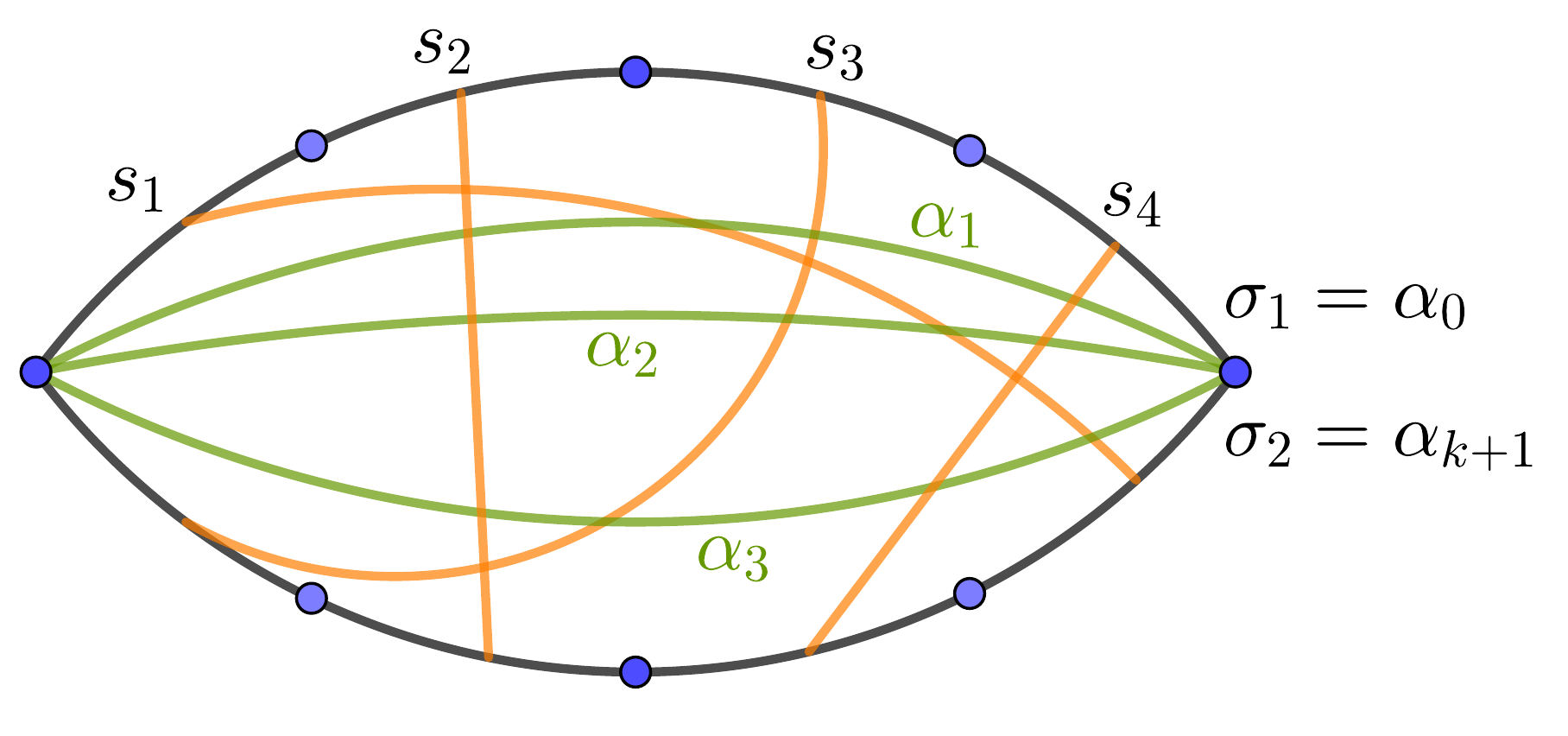}
\caption{$w_1= s_1s_2s_3s_4 \to s_2s_1s_3s_4 \to s_2s_3s_1s_4 \to s_2s_3s_4s_1 \to s_3s_2s_4s_1=w_2$.}
\label{word}
\end{center}
\end{figure}

\end{proof}

\begin{remark}
As a consequence of the uniqueness provided by Corollary \ref{cor:graphicallyreduced}, we will occasionally identify an element of a graph product with one of the graphically reduced words which represent it. For instance, in Section \ref{section:conj}, a graphically cyclically reduced element will refer to an element all of whose graphically reduced words which represent it are graphically cyclically reduced.
\end{remark}

\noindent
We conclude this section with a last consequence of Theorem \ref{thm:minlength}, which will also be useful in the next section. In the next statement, if $\Lambda$ is a subgraph of $\Gamma$, then $\langle \Lambda \rangle$ denotes the subgroup generated by the vertex-groups labelling the vertices of $\Lambda$.

\begin{cor}\label{cor:parabolic}
Let $\Gamma$ be a simplicial graph and $\mathcal{G}$ a collection of groups indexed by $V(\Gamma)$. If $\Lambda \subset \Gamma$ is an induced subgraph and if $\mathcal{H}$ denotes the subcollection $\{ G_u \mid u \in V(\Gamma) \} \subset \mathcal{G}$, then the canonical map $\Lambda \mathcal{H} \to \langle \Lambda \rangle$ induces an isomorphism. Moreover, if $\Gamma$ is finite and if all the vertex-groups are finitely generated, then the subgroup $\langle \Lambda \rangle$ is quasi-isometrically embedded. 
\end{cor}

\begin{proof}
The map $\Lambda \mathcal{H} \to \langle \Lambda \rangle$ clearly induces a morphism and is surjective. Moreover, it sends a graphically reduced word in $\Lambda \mathcal{H}$ to a graphically reduced word in $\Gamma \mathcal{G}$, so that it follows from Theorem \ref{thm:minlength} that this map must be injective. Thus, we have proved that our map induces an isomorphism.

\medskip \noindent
Now, suppose that $\Gamma$ is finite and that all the vertex-groups are finitely generated. For every $u \in V(\Gamma)$, fix a finite generating set $X_u$ of $G_u$, and set $X= \bigcup\limits_{u \in V(\Gamma)} X_u$. Notice that $X$ is finite generating set of $\Gamma \mathcal{G}$.

\begin{claim}\label{claim:length}
For every graphically reduced word $s_1 \cdots s_n$, one has
$$\| s_1 \cdots s_n \|_{X} = \sum\limits_{i=1}^n \|s_i\|_{X_{u_i}}$$
where $u_i$ is the vertex of $\Gamma$ such that $s_i \in G_{u_i}$ for every $1 \leq i \leq n$.
\end{claim}

\noindent
Let $r_1 \cdots r_m$ be a word written over $X$ which represents $s_1 \cdots s_n$ and which has minimal length. Now, think of $r_1 \cdots r_m$ as a word written over $\bigcup\limits_{u \in V(\Gamma)} G_u$. Apply iteratively the following operation: if there exist $1 \leq i <j \leq m$ such that $r_i$ and $r_j$ belong to the same vertex-group, which is adjacent to the vertex-group containing $r_k$ for every $i < k < j$. Eventually we obtain a graphically reduced word $p_1 \cdots p_k$ written over $\bigcup\limits_{u \in V(\Gamma)} G_u$ such that $\|s_1 \cdots s_n \| = \sum\limits_{i=1}^k \|p_i\|$. As a consequence of Corollary \ref{cor:graphicallyreduced}, there exists a bijection $\sigma : \{ 1, \ldots, k\} \to \{1, \ldots, n\}$ such that $p_i = s_{\sigma(i)}$ for every $1 \leq i \leq k$. Therefore,
$$\|s_1 \cdots s_n\| = \sum\limits_{i=1}^k \|p_i \| = \sum\limits_{i=1}^k \| s_{\sigma(i)} \| = \sum\limits_{i=1}^n \| s_i \|,$$
concluding the proof of our claim. 

\medskip \noindent
Set $Y= \bigcup\limits_{u \in V(\Lambda)} X_u$. Since the map $\Lambda \mathcal{H} \to \langle \Lambda \rangle$ sends a graphically reduced word of $\Lambda \mathcal{H}$ to a graphically reduced word of $\Gamma \mathcal{G}$, it follows from Claim \ref{claim:length} that the map $(\Lambda \mathcal{H}, \|\cdot\|_Y) \hookrightarrow (\Gamma \mathcal{G},\|\cdot\|)$ induced by $\Lambda \mathcal{H} \overset{\sim}{\longrightarrow} \langle \Lambda \rangle \hookrightarrow \Gamma \mathcal{G}$ is an isometric embedding. Therefore, the subgroup $\langle \Lambda \rangle$ is quasi-isometrically embedded. 
\end{proof}

\section{Dehn functions}\label{section:Dehn}

\noindent
In this section, our goal is to compute the Dehn function of a graph product of finitely presented groups. We begin by recalling basic definitions related to Dehn functions.

\begin{definition}
Let $\mathcal{P}= \langle X \mid R \rangle$ be a finite group presentation. If $w$ is a word written over $X$, the \emph{area} of $w$ is the smallest possible number of $2$-cells of a van Kampen diagram over $\mathcal{P}$ whose boundary is labelled by $w$. The \emph{Dehn function} of $\mathcal{P}$ is
$$\delta_{\mathcal{P}} : n \mapsto \max \{ \mathrm{Area}(w) \mid \|w\| \leq n \}.$$
\end{definition}

\noindent
But this defines the Dehn function of a group presentation and not of the group itself. Moreover, the Dehn function really depends on the presentation we choose. In order to remove this dependence, we introduce a specific equivalence relation between functions. 

\begin{definition}\label{def:equivalentfunctions}
Let $f,g : \mathbb{N} \to \mathbb{N}$ be two functions. We write $f \prec g$ if there exists a constant $A>0$ such that
$$f(n) \leq A \cdot g(An+A)+An+A$$
for every $n \geq 0$. The functions are \emph{equivalent}, written $f \sim g$, if $f \prec g$ and $g \prec f$. 
\end{definition}

\noindent
Now, a classical result states that the Dehn functions of all the finite presentations of a given group are pairwise equivalent. Namely:

\begin{prop}
Let $\mathcal{P}_1$ and $\mathcal{P}_2$ be two finite presentations of a group $G$, then $\delta_{\mathcal{P}_1} \sim~\delta_{\mathcal{P}_2}$. 
\end{prop}

\noindent
We refer to \cite{AlonsoDehn} for a proof of this statement. This proposition justifies the following definition:

\begin{definition}
The \emph{Dehn function} of a finitely presented group $G$ is the equivalence class of the Dehn functions of its finite presentations, written $\delta_G$.
\end{definition}

\noindent
If a group $G$ has a fixed finite presentation $\mathcal{P}$, then $\delta_G$ and $\delta_{\mathcal{P}}$ may be identified for convenience. 

\medskip \noindent
So, given a finite simplicial graph $\Gamma$ and a collection of finitely presented groups $\mathcal{G}$ indexed by $V(\Gamma)$, we want to identify the equivalence class of functions containing the Dehn functions of the finite presentations of $\Gamma \mathcal{G}$. However, the answer depends on $\Gamma$ and $\mathcal{G}$, which is why we introduce the following definition:

\begin{definition}
Let $\Gamma$ be a simplicial graph and $\mathcal{G}$ a collection of groups indexed by $V(\Gamma)$. The couple $(\Gamma, \mathcal{G})$ satisfies \emph{Meier's condition} if:
\begin{itemize}
	\item $\Gamma$ is square-free;
	\item no two infinite vertex-groups are adjacent;
	\item and the link of a vertex labelled by an infinite vertex-group is complete.
\end{itemize}
\end{definition}

\noindent
From now on, we fix a finite presentation $\mathcal{P}_G= \langle X_G \mid \mathcal{R}_G \rangle$ for each $G \in \mathcal{G}$, and we denote by $\delta_G$ the corresponding Dehn function. The Dehn function $\delta_{\Gamma \mathcal{G}}$ corresponds to the presentation 
$$\mathcal{P}= \left\langle \bigcup\limits_{G \in \mathcal{G}} X_G \ \left| \ \bigcup\limits_{G \in \mathcal{G}} \mathcal{R}_G, \ [x,y]=1 \ \text{if $x \in X_{G_1}$, $y \in X_{G_2}$ and $G_1,G_2 \in \mathcal{G}$ adjacent} \right. \right\rangle.$$ 
In order to state the main result of this section, the following notation will be needed. For every vertex group $G \in \mathcal{G}$, let $\widetilde{\delta_G}$ denote the Dehn function $\delta_G$ if the vertex associated to $G$ is adjacent to all the other vertices of $\Gamma$, or the the \emph{subnegative closure} $\overline{\delta_G}$ of $\delta_G$ otherwise. Given a function $f : \mathbb{N} \to \mathbb{N}$, its \emph{subnegative closure} $\overline{f}$ is defined as
$$\overline{f} : n \mapsto \max \left\{ \sum\limits_{i=1}^k f(n_i) \mid k \geq 1, \ \sum\limits_{i=1}^k n_i = n \right\}.$$
Notice that the subnegative closure of a function is always non-decreasing. Now we are ready to state our main result:

\begin{thm}\label{thm:Dehn}
Let $\Gamma$ be a finite simplicial graph and $\mathcal{G}$ a collection of finitely presented groups indexed by $V(\Gamma)$. 
\begin{itemize}
	\item If $\Gamma$ is a clique, then $\delta_{\Gamma \mathcal{G}} \sim \max( \delta_G, G \in \mathcal{G})$.
	\item If $(\Gamma, \mathcal{G})$ satisfies Meier's condition and is not a clique, then $$\delta_{\Gamma \mathcal{G}} \sim \max \left( n \mapsto n, \widetilde{\delta_G}, G \in \mathcal{G} \right)= \max \left( \overline{\delta_G}, G \in \mathcal{G} \right).$$
	\item If $(\Gamma, \mathcal{G})$ does not satisfy Meier's condition and is not a clique, then $$\delta_{\Gamma \mathcal{G}} \sim \max \left( n \mapsto n^2, \widetilde{\delta_G}, G \in \mathcal{G} \right).$$
\end{itemize}
\end{thm}

\noindent
The key result to prove the theorem will be the following statement:

\begin{prop}\label{prop:Dehn}
Let $\Gamma$ be a simplicial graph and $\mathcal{G}$ a collection of finitely presented groups indexed by $V(\Gamma)$. Then
$$\delta_{\Gamma \mathcal{G}}(n) \leq \sum\limits_{G \in \mathcal{G}} \overline{\delta_G}(n) + n^2$$
for every $n \geq 0$. 
\end{prop}

\begin{proof}
We say that a van Kampen diagram $\Delta$ over $\mathcal{Q}$ satisfies \emph{Condition $(\ast)$} if, for every dual curve $\gamma$ of $\Delta$ which is not a circle, 
\begin{itemize}
	\item the union of all the triangles crossed by $\gamma$ is homeomorphic to a disc,
	\item and there exists at least one triangle crossed by $\gamma$ which intersects $\partial D$ along an edge.
\end{itemize}
Let $w$ be a word of length at most $n$ written over $\bigcup\limits_{G \in \mathcal{G}} X_G$. 

\begin{claim}
There exists a van Kampen diagram over $\mathcal{Q}$ which satisfies Condition $(\ast)$ and whose boundary is labelled by $w$.
\end{claim}

\noindent
Let $\Delta$ be a van Kampen diagram over $\mathcal{Q}$ whose boundary is labelled by $w$, and let $\gamma$ be a dual curve of $\Delta$ which is not a circle. According to Corollary \ref{cor:dualtree}, $\gamma$ is a tree. If $T$ is a triangle crossed by $\gamma$ such that the corresponding singularity of $\gamma$ is linked to $\partial \Delta$ by an arc of $\gamma$, then it is possible to apply pentagonal moves in order to make $T$ intersect $\partial \Delta$ along an edge. See Figure \ref{triangleboundary}. Consequently, without loss of generality, we suppose from now on that $T \cap \partial \Delta$ contains an edge.
\begin{figure}
\begin{center}
\includegraphics[trim={0 17cm 21cm 0},clip,scale=0.4]{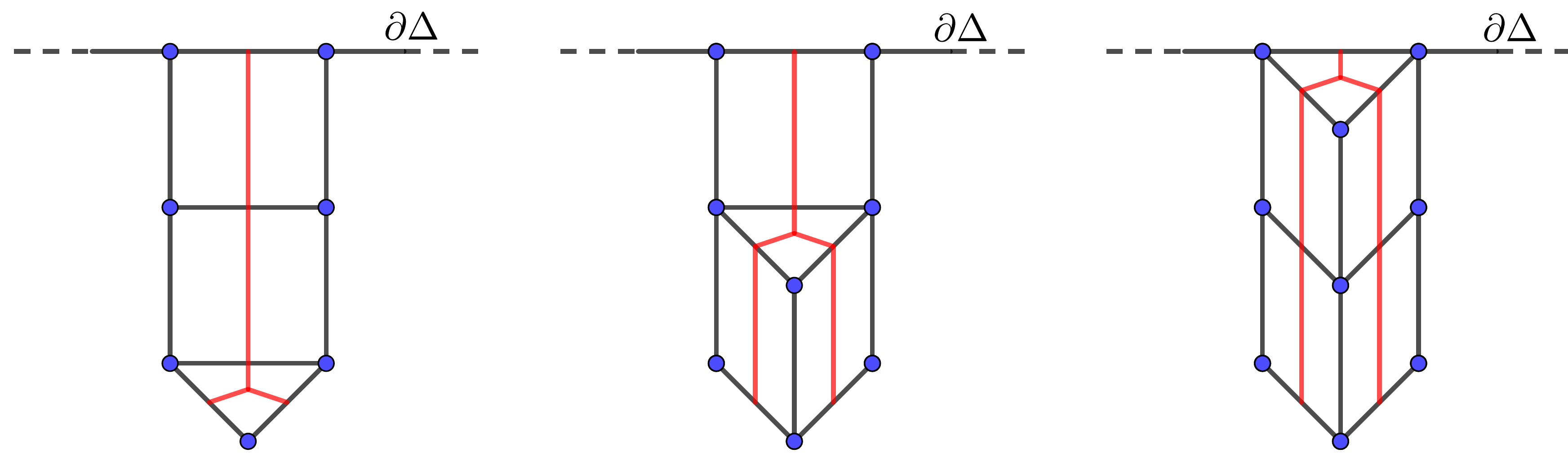}
\caption{Translating a triangle to the boundary.}
\label{triangleboundary}
\end{center}
\end{figure}

\medskip \noindent
Define the \emph{complexity of $\Delta$ with respect to $T$} as the sum of the numbers of squares crossed by the subsegments of $\gamma$ linking the singularity associated to $T$ to all the other singularities. If $\Delta$ has complexity $\geq 1$, then $\gamma$ contains an arc between two singularities which intersects at least one square. Let $T_1$ and $T_2$ denote the triangles corresponding to these two singularities. Without loss of generality, suppose that $T_1$ corresponds to the singularity which is the closest to the singularity associated to $T$ (with respect to the metric defined on the tree $\gamma$). By applying pentagonal moves, it is possible to translate $T_2$ to $T_1$ in order to make it intersect $T_1$ along an edge. See Figure \ref{triangletriangle}. Notice that the new diagram we get has lower complexity. Therefore, by iterating the process, it is possible to construct a new van Kampen diagram $\Delta'$ from $\Delta$ by pentagonal moves so that $\Delta'$ has complexity zero with respect to $T$. (Here, we identify the triangle $T$ of $\Delta$ with its image in $\Delta'$. This abuse of notation is justified by the fact that we did not move $T$ when applying our pentagonal moves.) Notice that, for any triangle $T'$ of $\Delta'$ crossed by the dual curve crossing $T$, there exists a sequence of triangles $T_0=T, T_1, \ldots, T_{n-1}, T_n=T'$ such that $T_i \cap T_{i+1}$ contains an edge for every $0 \leq i \leq n-1$. This observation implies that the union of all the triangles crossed by the dual curve crossing $T$ is contractible, and consequently it has to be homeomorphic to a disc. 
\begin{figure}
\begin{center}
\includegraphics[trim={0 16cm 37cm 0},clip,scale=0.5]{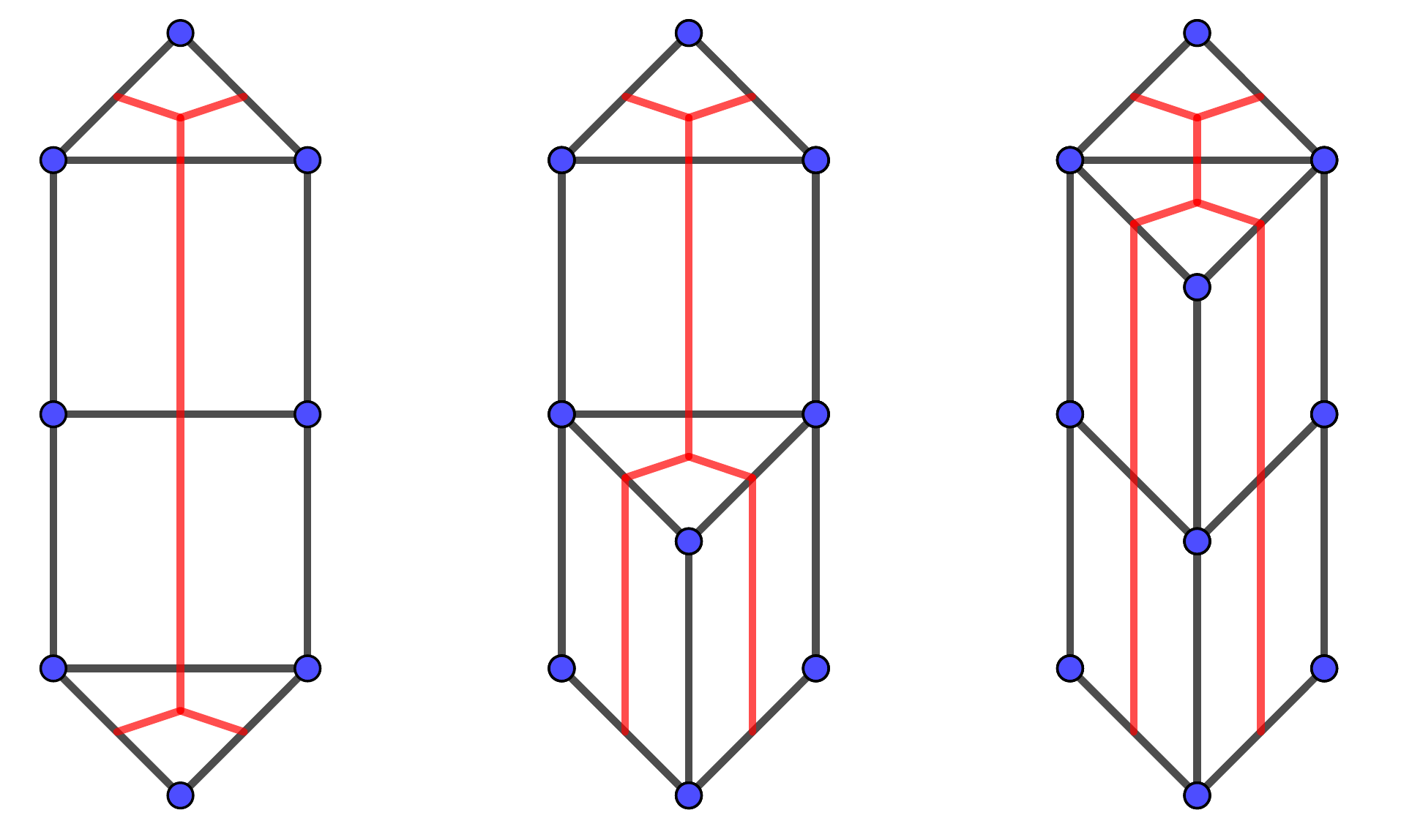}
\caption{Translating a triangle to another triangle.}
\label{triangletriangle}
\end{center}
\end{figure}

\medskip \noindent
Thus, we have constructed a new van Kampen diagram $\Delta'$ in which the dual curve corresponding to $\gamma$ satisfies the two conditions required by Condition $(\ast)$. Notice that, when constructing $\Delta'$ from $\Delta$, we did not move any triangle which is not crossed by $\gamma$. Consequently, we can iterate the process to another dual curve of $\Delta'$, and so on. The output of the construction is a van Kampen diagram whose boundary is labelled by $w$ and which satisfies Conditions $(\ast)$, concluding the proof of our claim.

\medskip \noindent
Let $\Delta$ be a van Kampen diagram whose boundary is labelled by $w$ and which has minimal area among all the diagrams satisfying Condition $(\ast)$. Let $\gamma_1, \ldots, \gamma_r$ denote the dual curves of $\Delta$; for every $1 \leq i \leq r$, let $U_i$ denote the union of all the triangles crossed by $\gamma_i$; and let $\Delta_0$ denote the union of all the squares of $\Delta$. See Figure \ref{fillin}. We construct a van Kampen diagram $\Delta^+$ over $\mathcal{P}$ whose boundary is labelled by $w$ from $\Delta$ by replacing each $U_i$ by a van Kampen diagram over $\mathcal{P}_G$ of minimal area with the same boundary (where $G$ is the vertex-group labelling the edges of $U_i$); and by replacing each square by a van Kampen diagram over $\langle X_{G_1}, X_{G_2} \mid \mathcal{R}_{G_1}, \mathcal{R}_{G_2}, [x,y]=1 \ \text{if $x \in X_{G_1}, y \in X_{G_2}$} \rangle$ of minimal area with the same boundary (where $G_1$ and $G_2$ are the two vertex-groups labelling the edges of the corresponding square). Our goal is to prove that the area of $\Delta^+$ is bounded above by~$n^2+ \sum\limits_{G \in \mathcal{G}} \overline{\delta_G}(n)$.
\begin{figure}
\begin{center}
\includegraphics[trim={0 19cm 29cm 0},clip,scale=0.5]{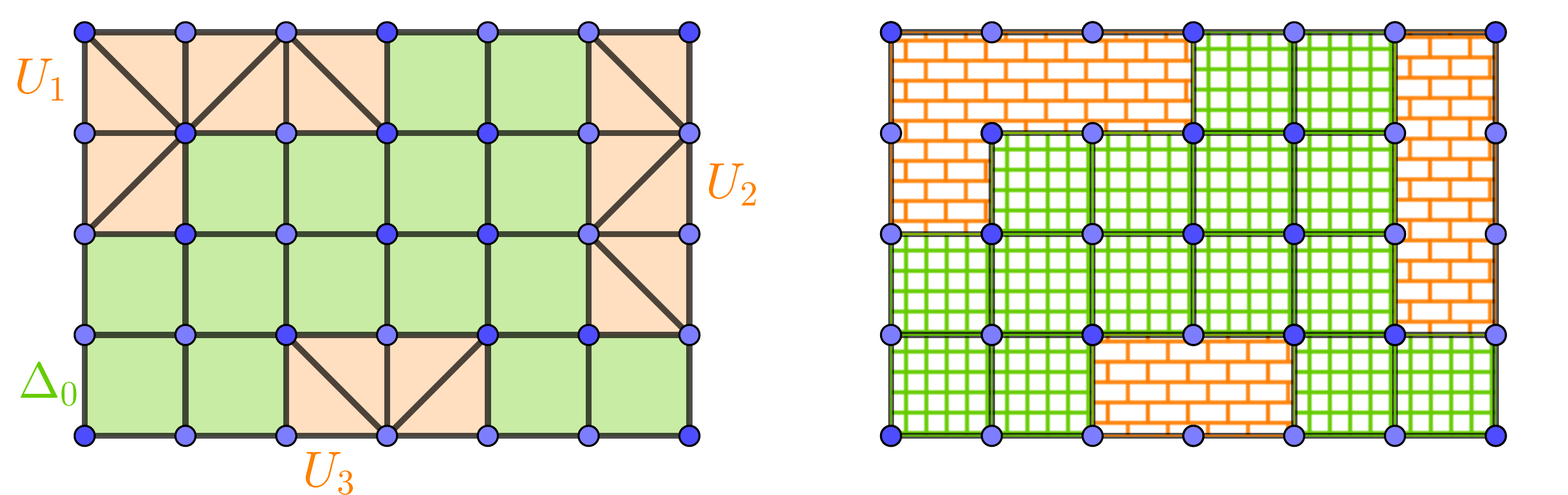}
\caption{A van Kampen diagram over $\mathcal{P}$ from a diagram over $\mathcal{Q}$ satisfying $(\ast)$.}
\label{fillin}
\end{center}
\end{figure}

\medskip \noindent
For every $u \in V(\Gamma)$, we denote by $I_u$ the set of indices $i$ such that $\gamma_i$ is labelled by the vertex-group $G_u$. We claim that $\displaystyle \sum\limits_{i \in I_u} \mathrm{Area}(U_i^+) \leq \overline{\delta_{G_u}}(n)$ where $U_i^+$ denotes the van Kampen diagram replacing $U_i$ when constructing $\Delta^+$ from $\Delta$. 

\medskip \noindent
For every $i \in I_u$, let $w_i=s_1^i \cdots s_{r(i)}^i$ denote the word labelling $\partial U_i$, and let $e^i_k$ denote the edge of $\partial U_i$ corresponding to the syllable $s_k^i$ for every $1 \leq k \leq r(i)$. Because there is an arc of $\gamma_i$ crossing both $e^i_k$ and $\partial \Delta$, it follows from Lemma \ref{lem:labelarc} that $s_k^i$ is also a syllable of $w$. Moreover, since an arc of a dual curve of $\Delta$ which intersects $\partial \Delta$ cannot cross two edges of some $U_i$ or two edges of two distinct $U_i$, then each $s_k^i$ corresponds to a distinct syllable of $w$. As a consequence, we deduce that
$$\sum\limits_{i \in I_u} \mathrm{Area}(U_i^+) \leq \sum\limits_{i \in I_u} \delta_{G_u}( \mathrm{length}(w_i)) \leq \overline{\delta_{G_u}}(n)$$
since $\displaystyle \sum\limits_{i \in I_u} \mathrm{length}(w_i) \leq \mathrm{length}(w)  \leq n$. This concludes the proof of our claim.

\medskip \noindent
Now, we want to prove that, if $\Delta_0^+$ denotes the subdiagram of $\Delta^+$ corresponding to $\Delta_0$, then $\mathrm{Area}(\Delta_0^+) \leq  n^2$. First of all, notice that:

\begin{claim}
A dual curve of $\Delta_0$ cannot be a circle, and two distinct dual curves intersect at most once.
\end{claim}

\noindent
It follows from the proof of \cite[Lemma 2.2]{BigWise} that, if $\Delta_0$ contains a dual curve which is a circle or two dual curves which intersect at least twice, then the area of $\Delta_0$ can be decreased by applying hexagonal moves and square-reductions. As a consequence, the minimality of the area of $\Delta$ implies that $\Delta_0$ cannot contain a dual curve which is a circle nor two dual curves intersecting twice.

\medskip \noindent
Let $\alpha_1, \ldots, \alpha_p$ denote the dual curves of $\Delta_0$. They are arcs of dual curves of $\Delta$, which have to intersect $\partial \Delta$. We know from Lemma \ref{lem:labelarc} that, for every $1 \leq i \leq p$, all the edges crossed by $\alpha_i$ (if we orient them in a coherent way) are labelled by the same syllable, say $q_i$. Notice that each $q_i$ has to be a syllable of $w$ since $\alpha_i$ intersects $\partial \Delta$. Hence
$$\mathrm{Area}(\Delta_0^+) \leq \sum\limits_{1 \leq i,j \leq p} \mathrm{length}(q_i) \cdot \mathrm{length}(q_j) = \left( \sum\limits_{i=1}^p \mathrm{length}(q_i) \right)^2 \leq \mathrm{length}(w)^2 \leq n^2,$$
which is the desired conclusion.

\medskip \noindent
Finally, we deduce from the previous two inequalities that
$$\mathrm{Area}(\Delta^+) = \mathrm{Area}(\Delta_0^+) + \sum\limits_{u \in V(\Gamma)} \sum\limits_{i \in I_u} \mathrm{Area}(U_i^+) \leq n^2+ \sum\limits_{G \in \mathcal{G}} \overline{\delta_G}(n),$$
which concludes the proof of our proposition.
\end{proof}

\noindent
Before turning to the proof of Theorem \ref{thm:Dehn} we need two last preliminary results, namely:

\begin{prop}\label{prop:DehnDirect}
Let $G$ and $H$ be two finitely presented groups. Then
$$\delta_{G \oplus H} \sim \left\{ \begin{array}{cl} \max(n \mapsto n^2, \delta_G, \delta_H) & \text{if $G$ and $H$ are infinite} \\ \max( \delta_G, \delta_H) & \text{otherwise}\end{array} \right..$$
\end{prop}

\begin{proof}
If $G$ (resp. $H$) is finite then $H$ (resp. $G$) is a finite-index subgroup of $G \oplus H$, hence $\delta_{G \oplus H} \sim \delta_H \sim \max(\delta_G, \delta_H)$ (resp. $\delta_{G \oplus H} \sim \delta_G \sim \max(\delta_G, \delta_H)$). From now on, we suppose that $G$ and $H$ are both infinite. The inequality $\delta_{G \oplus H} \prec  \max(n \mapsto n^2, \delta_G, \delta_H)$ is proved in \cite[Proposition 2.1]{BrickDehn} and the inequality $\max ( \delta_G, \delta_H) \prec \delta_{G \oplus H}$ is proved in \cite[Corollary 2.3]{BrickDehn}. It remains to show that $(n \mapsto n^2) \prec \delta_{G \oplus H}$. 

\medskip \noindent
Because $G$ and $H$ are both infinite, there exist quasi-isometric embeddings $[0,+ \infty) \hookrightarrow G$ and $[0,+ \infty) \hookrightarrow H$, from which we deduce that there exists a quasi-isometric embedding $[0,+ \infty)^2 \hookrightarrow G \oplus H$. Therefore,
$$(n \mapsto n^2) \sim \delta_{[0,+ \infty)^2} \prec \delta_{G \oplus H}$$
as desired. (Here $[0,+ \infty)^2$ is identified with the square complex obtained by tiling the plane with squares. We refer to \cite{AlonsoDehn} for the definition of Dehn functions of cellular complexes.) 
\end{proof}

\begin{prop}\label{prop:DehnFree}
Let $G$ and $H$ be two non-trivial finitely presented groups. Then
$$\delta_{G \ast H} \sim \max(\overline{\delta_G}, \overline{\delta_H}).$$
\end{prop}

\begin{proof}
The inequality $\delta_{G \ast H} \prec \max(\overline{\delta_G}, \overline{\delta_H})$ is proved in \cite[Proposition 3.1]{BrickDehn}. The inequality $\max(\overline{\delta_G}, \overline{\delta_H}) \prec \delta_{G \ast H}$ is proved in~\cite{DehnFreeProd}. 
\end{proof}

\noindent
We are finally ready to prove the main result of this section, namely Theorem \ref{thm:Dehn}.

\begin{proof}[Proof of Theorem \ref{thm:Dehn}.]
If $\Gamma$ is a clique, then $\Gamma \mathcal{G}$ coincides with the direct sum $\bigoplus\limits_{G \in \mathcal{G}} G$, so the equivalence $\delta_{\Gamma \mathcal{G}} \sim \max (\delta_G, G \in \mathcal{G})$ follows from Proposition \ref{prop:DehnDirect}. From now on, we suppose that $\Gamma$ is not a clique.

\medskip \noindent
Suppose first that $(\Gamma, \mathcal{G})$ does not satisfy Meier's condition. Decompose $\Gamma$ as a join $\Lambda \ast C$ where $C$ is a clique and where each vertex of $\Lambda$ is not adjacent to some other vertex of $\Lambda$. Set $\mathcal{H}= \{G_u \mid u \in V(\Lambda)\}$. Since $\Gamma \mathcal{G}$ coincides with $\bigoplus\limits_{u\in V(C)} G_u \oplus \Lambda \mathcal{H}$, it follows from Proposition \ref{prop:DehnDirect} that $\delta_{\Gamma \mathcal{G}} \sim \max ( \eta, \delta_{\Lambda \mathcal{H}}, \delta_{G_u}, u \in V(C))$ where $\eta$ is a function which is either bounded or quadratic. Next, we know from Proposition \ref{prop:Dehn} that $\delta_{\Lambda \mathcal{H}} \prec \max(n \mapsto n^2, \overline{\delta_{G_u}}, u \in V(\Lambda))$. Therefore,
$$\delta_{\Gamma \mathcal{G}} \prec \max \left( n \mapsto n^2, \overline{\delta_{G_u}}, u \in V(\Lambda), \delta_{G_u},u \in V(C) \right) = \max\left( n \mapsto n^2, \widetilde{\delta_{G}}, G \in \mathcal{G} \right).$$
Conversely, if $G$ is a vertex-group, then the inequality $\delta_G \prec \delta_{\Gamma \mathcal{G}}$ holds since vertex-groups are quasi-isometrically embedded according to Corollary \ref{cor:parabolic}. Moreover, if $H$ is a vertex-group which is not adjacent to $G$, then it follows from Corollary \ref{cor:parabolic} that the subgroup $\langle G,H \rangle$ is quasi-isometrically embedded and isomorphic to the free product $G \ast H$. We deduce from Proposition \ref{prop:DehnFree} that $\overline{\delta_G} \prec \delta_{G \ast H} \prec \delta_{\Gamma \mathcal{G}}$. Thus, we have proved that 
$$\max \left( \widetilde{\delta_G}, G \in \mathcal{G} \right) \prec \delta_{\Gamma \mathcal{G}}.$$ 
We emphasize that this inequality holds in full generality, ie., it does not depend on Meier's condition. However, this condition will imply that $(n \mapsto n^2) \prec \delta_{\Gamma \mathcal{G}}$. Indeed, 
\begin{itemize}
	\item if there exist two adjacent vertex-groups $G$ and $H$, then we deduce that $$(n \mapsto n^2) \prec \delta_{G \oplus H} \prec \delta_{\Gamma \mathcal{G}}$$ from Corollary \ref{cor:parabolic} and Proposition \ref{prop:DehnDirect};
	\item if $\Gamma$ contains a square $\Lambda$ whose vertices are labelled by the groups $A,B,C,D$, then $\langle \Lambda \rangle \simeq (A \ast C) \oplus (B \ast D)$ so that we deduce that $$(n \mapsto n^2) \prec \delta_{\langle \Lambda \rangle} \prec \delta_{\Gamma \mathcal{G}}$$ from Corollary \ref{cor:parabolic} and Proposition \ref{prop:DehnDirect};
	\item and if there exists a vertex $u \in V(\Gamma)$ labelled by an infinite group which is adjacent to two non-adjacent vertices $v,w \in V(\Gamma)$, then $\langle G_u,G_v,G_w \rangle \simeq G_u \oplus (G_v \ast G_w)$ so that we deduce that $$(n \mapsto n^2) \prec \delta_{\langle G_u,G_v,G_w \rangle} \prec \delta_{\Gamma \mathcal{G}}$$ from Corollary \ref{cor:parabolic} and Proposition \ref{prop:DehnDirect}.
\end{itemize}
Now, suppose that $(\Gamma, \mathcal{G})$ satisfies Meier's condition. Because $\Gamma$ is not a clique, it has to contain two non-adjacent vertices, say $u$ and $v$. Then it follows from Corollary \ref{cor:parabolic} that $(n \mapsto n) \prec \delta_{\langle G_u,G_v \rangle} \prec \delta_{\Gamma \mathcal{G}}$ since $\langle G_u, G_v \rangle \simeq G_u \ast G_v$ is infinite. Moreover, we saw above that $\max \left( \widetilde{\delta_G}, G \in \mathcal{G} \right) \prec \delta_{\Gamma \mathcal{G}}$. Therefore, we know that
$$\max \left( n \mapsto n, \widetilde{\delta_G}, G \in \mathcal{G} \right) \prec \delta_{\Gamma \mathcal{G}}.$$
Notice that, if $G$ is a vertex-group which is adjacent to all the other vertex-groups, then Meier's condition and the fact that $\Gamma$ is not a clique implies that $G$ must be finite. Otherwise, saying $G$ is a finite direct factor of $\Gamma \mathcal{G}$. This observation implies that
$$\max \left(n \mapsto n , \widetilde{\delta_G} , G \in \mathcal{G} \right) = \max \left( \overline{\delta_G}, G \in \mathcal{G} \right),$$
and also that we may suppose without loss of generality that $\Gamma$ does not contain any vertex which is adjacent to all the other vertices. Our goal now is to prove the inequality $\delta_{\Gamma \mathcal{G}} \prec \max \left( \overline{\delta_G}, G \in \mathcal{G} \right)$. 

\medskip \noindent
Let $u \in V(\Gamma)$ be a vertex labelled by an infinite vertex-group. We deduce from Meier's condition that $\mathrm{link}(u)$ is a complete graph whose vertices are labelled by finite vertex-groups. As a consequence, $\Gamma \mathcal{G}$ decomposes an amalgamated sum over a finite subgroup, namely $\langle \Gamma \backslash \{u \} \rangle \underset{\langle \mathrm{link}(u) \rangle}{\ast} \langle \mathrm{star}(u) \rangle$. It follows from \cite[Proposition 3.2]{BrickDehn} that 
$$\delta_{\Gamma \mathcal{G}} \prec \max \left( \overline{\delta_{\langle \Gamma \backslash \{u \} \rangle}}, \overline{\delta_{\langle \mathrm{star}(u) \rangle}} \right) \sim \max \left( \overline{\delta_{\langle \Gamma \backslash \{u \} \rangle}}, \overline{\delta_{G_u}} \right)$$
where we used the equivalence $\delta_{\langle \mathrm{star}(u) \rangle} \sim \delta_{G_u}$ which implied by the fact that $G_u$ has finite index in $\langle \mathrm{star}(u) \rangle$. By iterating the argument to all the infinite vertex-groups, it follows that
$$\delta_{\Gamma \mathcal{G}} \prec \max \left( \overline{\delta_{\langle \Lambda \rangle}}, \overline{\delta_{G_u}}, u \in I \right) = \max \left( \overline{\delta_{\langle \Lambda \rangle}} , \overline{\delta_G}, G \in \mathcal{G} \right)$$
where $\Lambda$ denotes the subgraph of $\Gamma$ induced by the vertices labelled by finite vertex-groups and where $I \subset V(\Gamma)$ denotes the vertices labelled by infinite vertex-groups. As a consequence of Corollary \ref{cor:parabolic}, the subgroup $\langle \Lambda \rangle$ is isomorphic to the graph product $\Lambda \mathcal{H}$ where $\mathcal{H}= \{ G_u \mid u \in V(\Lambda) \}$. Notice that $(\Lambda, \mathcal{H})$ satisfies Meier's condition, so that we know from \cite{MeierGP} (see also \cite[Section 8.3]{Qm}) that $\Lambda \mathcal{H}$ is a hyperbolic group, so that $\overline{\delta_{\langle \Lambda \rangle}}$ must be linear. Thus, we have proved
$$\delta_{\Gamma \mathcal{G}} \prec \max \left( \overline{\delta_G}, G \in \mathcal{G} \right)$$
concluding the proof of our theorem.
\end{proof}

\begin{remark}
Whether there exists a finitely presented group whose Dehn function is not equivalent to a subnegative function is still an open question. In \cite{DehnFreeProd}, it is proved that the Dehn function of a free product is always subnegative. As a consequence of Theorem \ref{thm:Dehn}, we find more generally that the Dehn function of a graph product which is not a direct sum is always subnegative.
\end{remark}

\section{Conjugacy problem}\label{section:conj}

\noindent
In this section, we are interested in the conjugacy problem in graph products of groups. More precisely, we would like to be able to determine whether or not two given elements of a graph product are conjugate. We begin by introducing \emph{graphically reduced words}, which is a generalisation of cyclically reduced words in free groups and more generally in free products. 

\begin{definition}
Let $\Gamma$ be a simplicial graph and $\mathcal{G}$ a collection of groups indexed by $V(\Gamma)$. A word $s_1 \cdots s_n$ written over $\bigcup\limits_{u \in G_u} G_u$ is \emph{graphically cyclically reduced} if it is graphically reduced and if there does not exist $1 \leq i< j \leq n$ such that the vertex-group containing $s_i$ (resp. $s_j$) is adjacent to the vertex-group containing $s_k$ for every $1 \leq k < i$ (resp. for every $j<k \leq n$).
\end{definition}

\noindent
It is worth noticing that a power of graphically cyclically reduced word may not be graphically reduced. This phenomenon is due to the possible existence of \emph{floating syllables}.

\begin{definition}
Let $\Gamma$ be a simplicial graph, $\mathcal{G}$ a collection of groups indexed by $V(\Gamma)$, and $s_1 \cdots s_n$ a word written over $\bigcup\limits_{u \in G_u} G_u$. For every $1 \leq i \leq n$, the syllable $s_i$ is \emph{floating} if the vertex-group which contains $s_i$ is adjacent to the vertex-group containing $s_k$ for every $1 \leq k \leq n$ distinct from $i$. 
\end{definition}

\noindent
Our first step is to show that the conjugacy problem may be reduced to graphically cyclically reduced elements:

\begin{lemma}\label{lem:graphicallycyclicallyreduced}
Let $\Gamma$ be a simplicial graph and $\mathcal{G}$ a collection of groups indexed by $V(\Gamma)$. For every $a \in \Gamma \mathcal{G}$, there exist $b,g \in \Gamma \mathcal{G}$ such that $a=gbg^{-1}$ and such that $b$ is graphically cyclically reduced.
\end{lemma}

\begin{proof}
Write $a$ as a graphically reduced word $a_1 \cdots a_n$. If this word is not graphically cyclically reduced, then there exist $1 \leq i <j \leq n$ such that $a_i$ and $a_j$ belong to a common vertex-group and such that $a_i$ shuffles to the beginning in $a_1 \cdots a_i$ and $a_j$ shuffles to the end in $a_j \cdots a_n$. Therefore, 
$$\begin{array}{lcl} a & = & a_i a_1 \cdots a_{i-1}a_{i+1} \cdots a_{j-1}a_{j+1} \cdots a_n a_j \\ \\ & = & a_i \cdot \underset{=:a'}{\underbrace{\left( a_1 \cdots a_{i-1}a_{i+1} \cdots a_{j-1}a_{j+1} \cdots a_n (a_ja_i) \right)}} \cdot a_i^{-1} \end{array}$$
in $\Gamma \mathcal{G}$. Notice that the word $a'$ is shorter than the word $a$. As a consequence, by iterating the argument, we eventually find a graphically cyclically reduced element which is conjugate to $a$.
\end{proof}

\begin{remark}\label{remark}
It is worth noticing that the proof is constructive, so that there exists an algorithm to find $b$ and $g$ from $a$ if the word problem is solvable in all the vertex-groups (which is necessary in order to determine if a word is graphically reduced).
\end{remark}

\noindent
We are now ready to state and prove the main result of this section:


\begin{thm}\label{thm:conjugacy}
Let $a,b \in \Gamma \mathcal{G}$ be two graphically cyclically reduced elements. Then $a$ and $b$ are conjugate if and only if there exist graphically reduced words $x_1 \cdots x_r p_1 \cdots p_n$ and $y_1 \cdots y_r q_1 \cdots q_n$ such that
\begin{itemize}
	\item $a=x_1 \cdots x_r p_1 \cdots p_n$ and $b=y_1 \cdots y_r q_1 \cdots q_n$;
	\item the $p_i$'s and the $q_i$'s are floating syllables of $a$ and $b$ respectively, and there exists a permutation $\sigma \in S_n$ such that $p_i$ and $q_{\sigma(i)}$ are conjugate in a vertex-group;
	\item $y_1 \cdots y_s$ can be obtained from $x_1 \cdots x_r$ by permuting two consecutive syllables which belong to adjacent vertex-groups and by performing cyclic permutations. 
\end{itemize} 
\end{thm}

\begin{proof}
First of all, suppose that the three conditions hold. So, for every $1 \leq i \leq n$, there exists an element $c_i$ which belongs to the same vertex-group as the one containing $p_i$ and $q_{i}$ and which conjugates these two elements, ie., $c_ip_ic_i^{-1}=q_{i}$. Also, there exists a word $z$ written over $\{x_1^{\pm 1}, \ldots, x_r^{\pm} \}$ such that $z \cdot x_1 \cdots x_r \cdot z^{-1}= y_1 \cdots y_s$. Setting $c=z c_1 \cdots c_n$, we have
$$\begin{array}{lcl} cac^{-1} & = & (zc_1 \cdots c_n) \cdot x_1 \cdots x_r p_1 \cdots p_r \cdot (zc_1 \cdots c_n)^{-1} \\ \\ & = & \left( z \cdot x_1 \cdots x_r \cdot z^{-1} \right) \cdot \left( c_1p_1c_1^{-1} \cdots c_np_nc_n^{-1} \right) \\ \\ & = &  y_1 \cdots y_s q_1 \cdots q_n = b \end{array}$$
Thus, we have proved that $a$ and $b$ are conjugate in $\Gamma \mathcal{G}$. 

\medskip \noindent
Conversely, suppose that $a$ and $b$ are conjugate. Fix an annular diagram $D$ whose inner and outer circles are labelled by $a$ and $b$ respectively.

\begin{claim}\label{claim:boundaryhyp}
Two edges of the same component of $\partial D$ cannot be dual to the same dual curve.
\end{claim}

\noindent
Suppose for contradiction that there exists a dual curve $\gamma$ of $D$ intersecting twice either $\partial_\text{inn} D$ or $\partial_\text{out}D$. We will assume that the latter case occurs, the other case being symmetric. Recall that $\partial_\text{out} D$ is labelled by $b$, which we write as a graphically cyclically reduced word $b_1 \cdots b_n$. Let $\sigma \subset \partial_\text{out} D$ denote a subsegment delimited by $\gamma$ which is included into a connected component of $D \backslash \gamma$ which does not contain $\partial_\text{inn}D$. If there exists a dual curve $\alpha$ intersecting $\sigma$ twice, then replace $\gamma$ with $\alpha$, and iterate. Without loss of generality, we may suppose that every dual curve of $D$ intersects $\sigma$ at most once. As a consequence, any dual curve intersecting $\sigma$ has to be transverse to $\gamma$. It follows from Lemmas \ref{lem:labelarc} and \ref{lem:labeltransverse} that either the basepoint of $\partial_\text{out}D$ does not belong to $\sigma$, so that there exist two indices $1 \leq i < j \leq n$ such that $b_j$ shuffles to the beginning in the word $b_{i+1} \cdots b_j$ and such that $r_i$ and $r_j$ belong to a common vertex-group; or the basepoint of $\partial_\text{out}D$ belongs to $\sigma$, so that there exist two indices $1 \leq i <j \leq n$ such that $b_i$ shuffles to the beginning in the word $b_1 \cdots b_i$, such that $b_j$ shuffles to the end in the word $b_j \cdots b_n$, and such that $b_i$ and $b_j$ belong to a common vertex-group. In both cases, we get a contradiction with the fact that the word $b_1 \cdots b_n$ is graphically cyclically reduced, concluding the proof of our claim. 

\medskip \noindent
Let us say that a dual curve of $D$ is \emph{circular} if it is a disjoint union of $k \geq 1$ embedded circles containing exactly two singularities together with $k+1$ arcs linking two consecutive circles, the outer circle to $\partial_{\text{out}} D$ and the inner circle to $\partial_{\text{inn}}D$. See Figure \ref{figure3}.
\begin{figure}
\begin{center}
\includegraphics[trim={0 1.5cm 30cm 0},clip,scale=0.25]{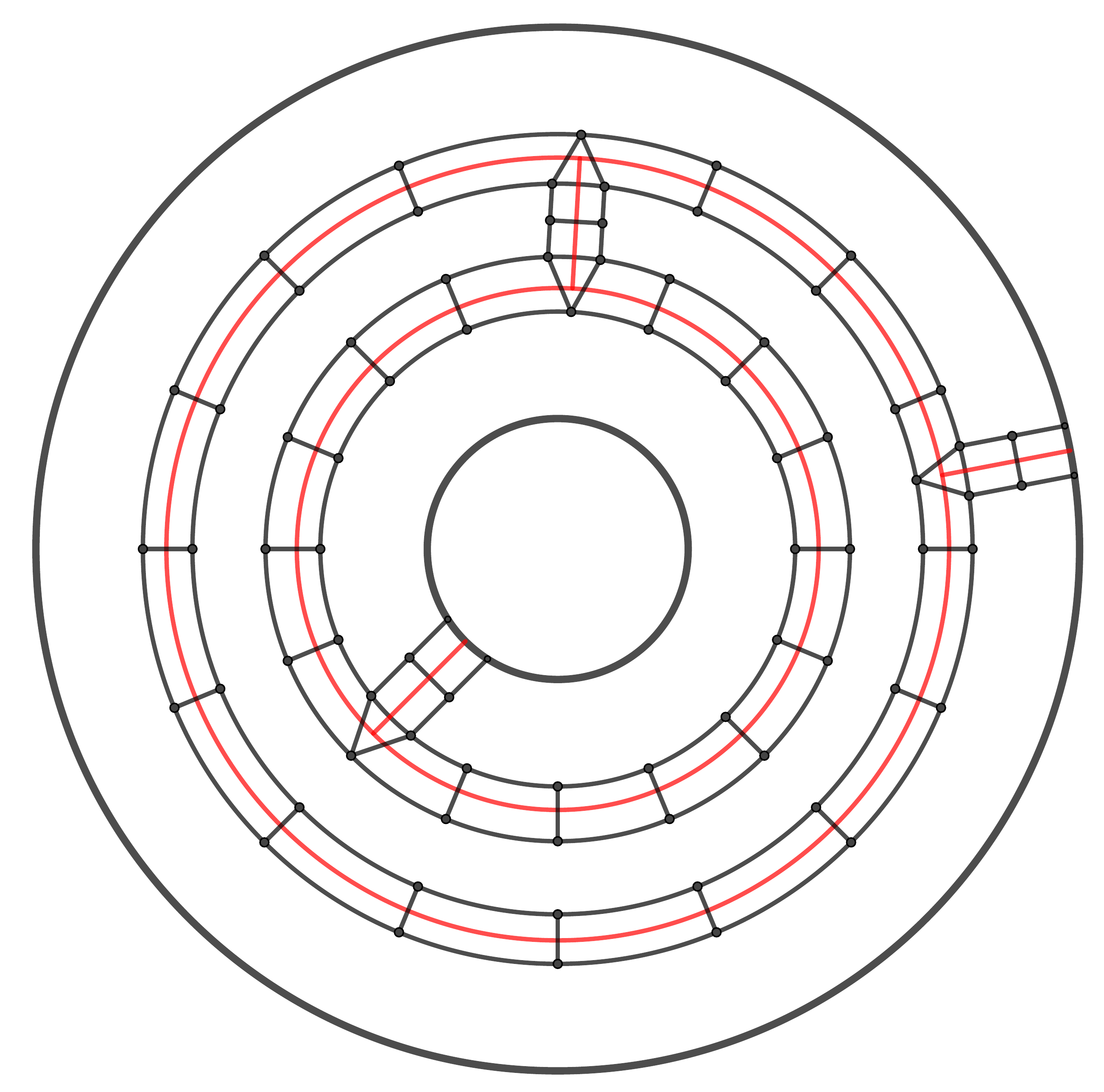}
\caption{A circular dual curve.}
\label{figure3}
\end{center}
\end{figure}

\begin{claim}
A dual curve intersecting $\partial D$ is either regular or circular.
\end{claim}

\noindent
Let $\gamma$ be a dual curve intersecting $\partial D$. Suppose that $\gamma$ is not regular. As a consequence, if $\gamma$ does not contain a circle, then it has to be a tree with at least three leaves, so that it has to intersect twice either $\partial_{\text{inn}} D$ or $\partial_{\text{out}} D$, contradicting Claim \ref{claim:boundaryhyp}. Therefore, $\gamma$ contains at least one circle. Let $C_1,C_2 \subset \gamma$ be two embedded circles which intersect. As a consequence of Proposition \ref{prop:circleindiag}, $D \backslash (C_1 \cup C_2)$ has to contain only two connected components: one containing $\partial_{\text{inn}} D$ and the other containing $\partial_{\text{out}} D$. By applying Euler's formula to the graph $C_1 \cup C_2$, it follows that its Euler characteristic has to be equal to zero. Otherwise saying, $C_1 \cup C_2$ can be constructed from a tree by adding an edge. Since $C_1 \cup C_2$ is leafless, the only possibility is that $C_1 \cup C_2$ is a circle, ie., $C_1=C_2$. Thus, we have proved that any two distinct embedded circles in $\gamma$ are disjoint. Let $C_1, \ldots, C_n$ denote the embedded circles of $\gamma$. As a consequence of Proposition \ref{prop:circleindiag}, we may index our circles so that $C_i$ separates $C_j$ from $\partial_\text{out} D$ for every $1 \leq i < j \leq n$. We suppose that $\gamma$ intersects $\partial_\text{out} D$, the case where $\gamma$ intersects $\partial_\text{inn}D$ being symmetric. So we know that $\gamma$ contains an arc connecting $\partial_\text{out} D$ to $C_1$, and as a consequence of Claim \ref{claim:boundaryhyp} there does not exist two such arcs. Next, because $\gamma$ is connected, for every $1 \leq i \leq n-1$ there must exist an arc connecting $C_i$ to $C_{i+1}$, and such an arc has to be unique according to Proposition \ref{prop:circleindiag}. Finally, since $C_n$ cannot contain a single singularity according to Claim \ref{claim:singlesing}, we know that $\gamma$ contains an arc connecting $C_n$ to $\partial_\text{inn} D$. Thus, we have proved that $\gamma$ is circular, concluding the proof of our claim. 

\medskip \noindent
Notice that a circular dual curve is transverse to any dual curve intersecting $\partial D$. It follows from Lemma \ref{lem:labeltransverse} that we can write $a$ as a graphically reduced word $x_1 \cdots x_r p_1 \cdots p_n$ where the $p_i$'s are the labels of the edges of $\partial_\text{out} D$ which are dual to circular dual curves (they are floating syllables, but a priori an $x_i$ may also be a floating syllable) and such that $x_1, \ldots, x_r$ appear in that order in the cyclic orientation of $\partial_\text{out}D$. Similarly, we can write $b$ as a graphically reduced word $y_1 \cdots y_s q_1 \cdots q_m$ where the $q_i$'s are the labels of the edges of $\partial_\text{in} D$ which are dual to circular dual curves and such that $y_1, \ldots, y_s$ appear in that order in the cyclic orientation of $\partial_\text{inn}D$. For every $1 \leq i \leq n$, let $\sigma(i)$ be such that $q_{\sigma(i)}$ is the label of the edge of $\partial_\text{in}D$ which is dual to the circular dual curve dual to the edge of $\partial_\text{out} D$ labelled by $p_i$. By construction, the map $p_i \mapsto q_{\sigma(i)}$ defines a bijection from $\{p_1, \ldots, p_n\}$ to $\{q_1, \ldots, q_m\}$. As a consequence, the indices $n$ and $m$ have to be equal. For every $1 \leq i \leq n$, $p_i$ and $q_{\sigma(i)}$ belong to the same vertex-group, say $G_i$, because they label two edges crossed by the same dual curve. Because the $q_i$'s pairwise commute, we will suppose without loss of generality that $\sigma(i)=i$ for every $1 \leq i \leq n$. 

\begin{claim}\label{claim:FloatingConj}
For every $1 \leq i \leq n$, $p_i$ and $q_{i}$ are conjugate in $G_i$.
\end{claim}

\noindent
We are claiming that the two elements labelling the arcs of a circular dual curve which intersect $\partial_\text{inn}D$ and $\partial_\text{out}D$ are conjugate in the corresponding vertex-group. It is sufficient to consider the case where the dual curve contains a single embedded circle, the general case following by induction on the number of embedded circles. Otherwise saying, we have to prove that the elements $c$ and $d$ in Figure \ref{figure4} are conjugate. But we have $d=a^{-1}b$ and $c=ba^{-1}$, hence $b^{-1} \cdot c \cdot b=d$. So our claim is proved.
\begin{figure}
\begin{center}
\includegraphics[trim={0 12cm 41cm 0},clip,scale=0.45]{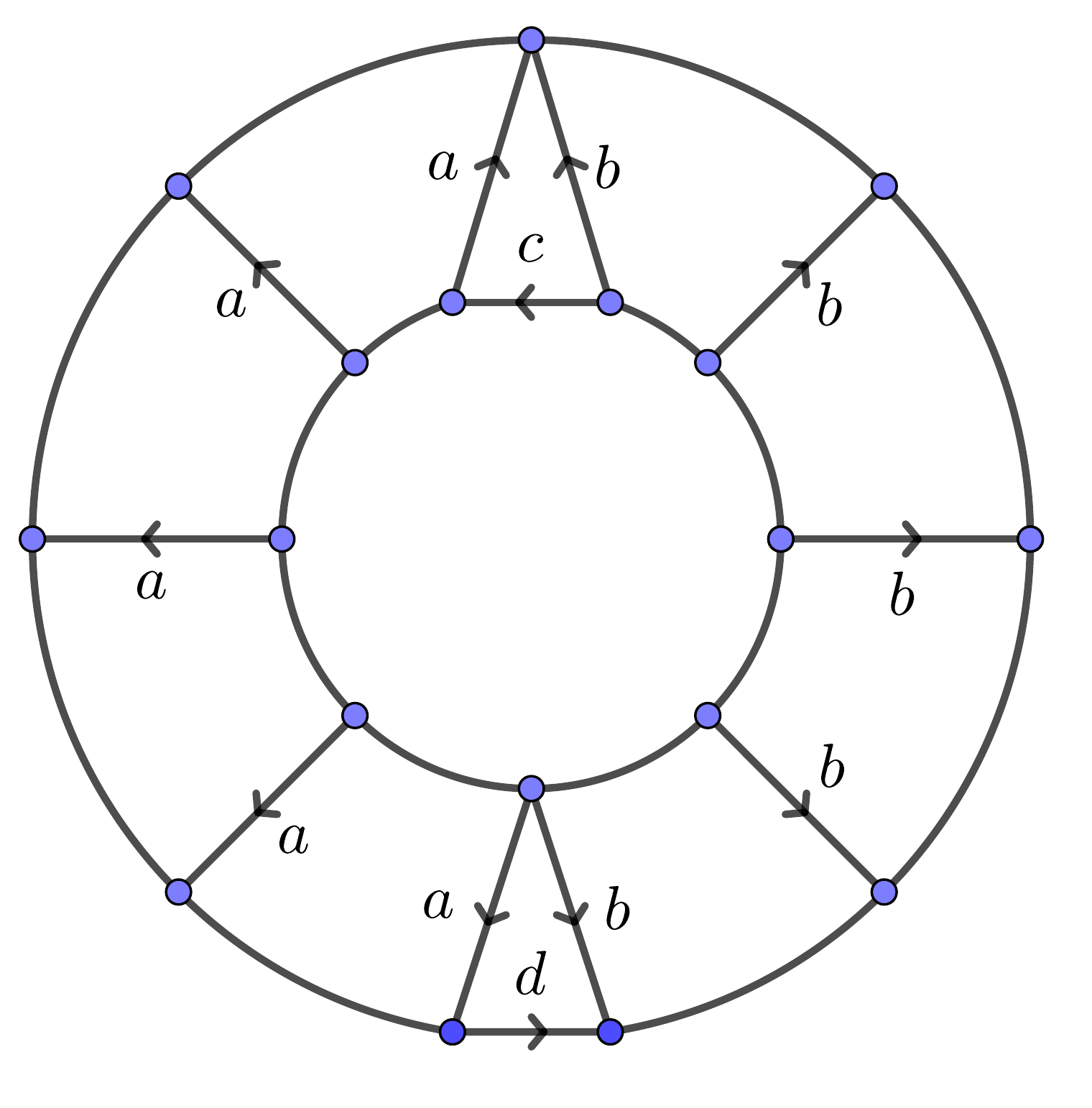}
\caption{Configuration used in the proof of Claim \ref{claim:FloatingConj}.}
\label{figure4}
\end{center}
\end{figure}


\begin{claim}\label{claim:permutation}
The word $y_1 \cdots y_s$ can be obtained from $x_1 \cdots x_r$ by permuting two consecutive syllables which belong to adjacent vertex-groups and by performing cyclic permutations. 
\end{claim}

\noindent
Let $C_1, \ldots, C_k$ be circles around $\partial_\text{inn} D$ such that, if we set $C_0=\partial_\text{inn}D$ and $C_{k+1}= \partial_\text{out} D$, then
\begin{itemize}
	\item each $C_i$ intersects the regular dual curve of $D$ transversely;
	\item the subspace delimited by $C_i$ and $C_{i+1}$ contains exactly one intersection point between two regular curves. 
\end{itemize}
We fix a basepoint on each $C_i$; for $i=0$ and $i=k+1$, we assume that this basepoint coincides with the basepoint which already exists. See Figure \ref{cyclicword}. For every $0 \leq i \leq k+1$, let $m_i$ denote the word obtained by reading the labels of the dual curves intersecting $C_i$. In particular, $m_0=x_1 \cdots x_r$ and $m_{k+1}= y_1 \cdots y_r$. Fix an index $0 \leq i \leq k$. If the basepoint of $C_i$ is not between the two dual curves which intersect between $C_i$ and $C_{i+1}$, then $m_{i+1}$ is obtained from $m_i$ by permuting two consecutive letters which label two transverse dual curves and next by performing a cyclic permutation on the word we obtain (due to the basepoint of $C_{i+1}$). It follows from Lemma \ref{lem:labeltransverse} that the first permutation amounts to permute two consecutive syllables which belong to two adjacent vertex-groups. Next, if the basepoint  of $C_i$ is between the two dual curves which intersect between $C_i$ and $C_{i+1}$, then we can write $m_i=amb$, where $m$ is some word and where $a$ and $b$ are the two letters labelling the dual curves which intersect between $C_i$ and $C_{i+1}$, such that $m_{i+1}$ is a cyclic permutation of $bma$ (due to the basepoint of $C_{i+1}$). Notice that
$$m_i=amb \to bam \to abm \to bma$$
are cyclic permutations and permutations of two consecutive syllables which belong to adjacent vertex-groups (once again as a consequence of Lemma \ref{lem:labeltransverse}). The proof of our claim is complete.
\begin{figure}
\begin{center}
\includegraphics[trim={0 9cm 38cm 0},clip,scale=0.45]{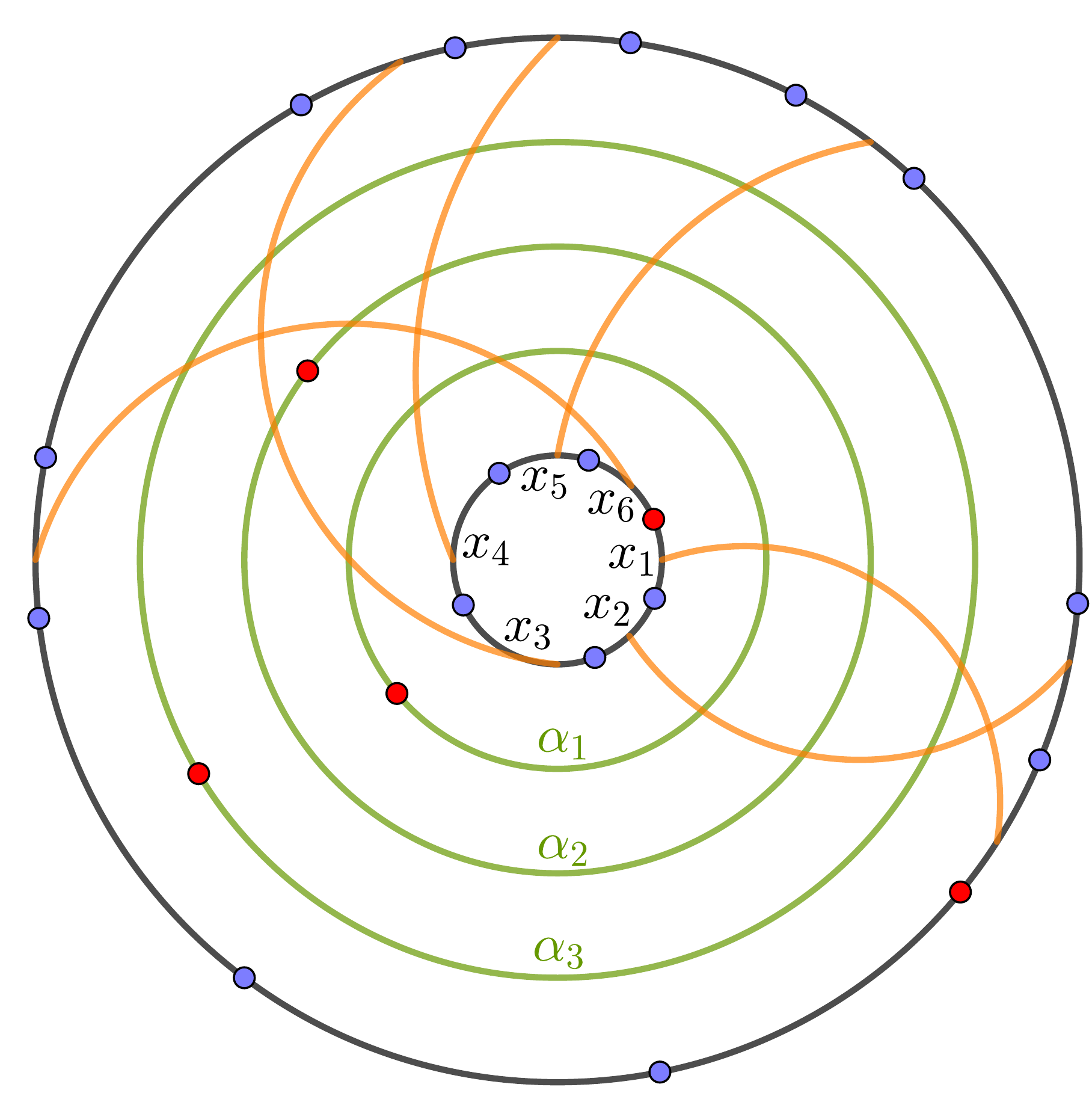}
\caption{The words labelling the circles are $m_0=x_1x_2x_3x_4x_5x_6$, $m_1= x_3x_4x_6x_5x_1x_2$, $m_2= x_6x_4x_5x_1x_2x_3$, $m_3= x_6x_3x_4x_5x_1x_2$, and $m_4= x_6x_3x_4x_5x_2x_1$.}
\label{cyclicword}
\end{center}
\end{figure}

\medskip \noindent
The combination of Claims \ref{claim:FloatingConj} and \ref{claim:permutation} concludes the proof of our theorem.
\end{proof}

\noindent
As an application of Theorem \ref{thm:conjugacy}, let us show that having a solvable conjugacy problem is a property which is stable under graph products. 

\begin{cor}
Let $\Gamma$ be a simplicial graph and $\mathcal{G}$ a collection of groups indexed by $V(\Gamma)$. Assume that the conjugacy problem is solvable in every vertex-group. Then the conjugacy problem is also solvable in $\Gamma \mathcal{G}$. 
\end{cor}

\begin{proof}
Let $a_0,b_0 \in \Gamma \mathcal{G}$ be two elements. As a consequence of Lemma \ref{lem:graphicallycyclicallyreduced} (see also Remark \ref{remark}), there exists an algorithm producing two graphically cyclically reduced elements $a$ and $b$ which are conjugate to $a_0$ and $b_0$ respectively. If $a$ and $b$ do not have the same number of floating syllables, then Theorem \ref{thm:conjugacy} implies that $a_0$ and $b_0$ are not conjugate. Otherwise, let $N$ denote their common number of floating syllables. For every $0 \leq n \leq N$, pick up $n$ floating syllables $p_1, \ldots, p_n$ of $a$, and for every $1 \leq i \leq n$, find a floating syllable $q_i$ of $b$ such that $p_i$ and $q_i$ are conjugate in a vertex-group thanks to the available algorithms which solve the conjugacy problems in vertex-groups. (If two such $n$-tuples do not exist, then Theorem \ref{thm:conjugacy} implies that $a_0$ and $b_0$ are not conjugate.) Write $a$ and $b$ respectively as graphically reduced words $a_1 \cdots a_r p_1 \cdots p_n$ and $b_1 \cdots b_s q_1 \cdots q_n$. Check if $a_1 \cdots a_r$ can be constructed from $b_1 \cdots b_s$ by permuting consecutive syllables which belong to adjacent vertex-groups and by performing cyclic permutations. If it is the case, then $a_0$ and $b_0$ are conjugate. Otherwise, repeat the operation with another tuples of floating syllables. When all the possibility have been considered without concluding that $a_0$ and $b_0$ are conjugate, we deduce from Theorem \ref{thm:conjugacy} that $a_0$ and $b_0$ are not conjugate.
\end{proof}

\noindent
Let us illustrate the solution to the conjugacy problem on a concrete example.

\begin{ex}\label{ex:conj}
Let $\Gamma \mathcal{G}$ be the graph product illustrated by Figure \ref{ex}. We claim that the elements $g=eab^2c^2ae^{-1}$ and $h=e^{-1}ba^2c^3eb$ are conjugate in $\Gamma \mathcal{G}$. First, notice that
$$g = e a^{-1} \cdot a^2b^2c^2 \cdot a e^{-1} \ \text{and} \ h= e^{-1}b^{-1} \cdot b^2a^2c^3 \cdot be$$
where $a^2b^2c^2$ and $b^2a^2c^3$ are graphically cyclically reduced. Next, notice that $c^2$ and $c^3$ are floating syllables of $a^2b^2c^2$ and $b^2a^2c^3$ respectively, that $a^2b^2$ is a cyclic permutation of $b^2a^2$, and that $c^2$ and $c^3$ are conjugate in $C$. The desired conclusion follows.
\begin{figure}
\begin{center}
\includegraphics[trim={0 21cm 39cm 0},clip,scale=0.45]{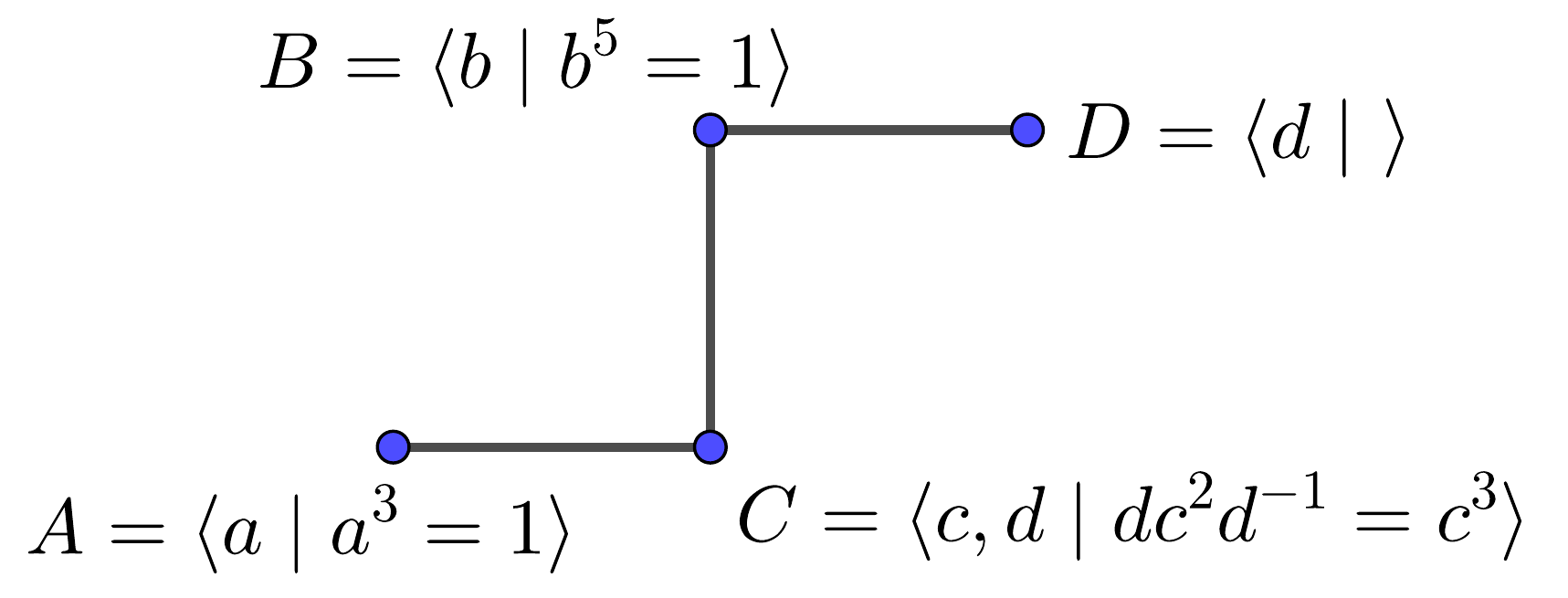}
\caption{The graph products from Example \ref{ex:conj}.}
\label{ex}
\end{center}
\end{figure}
\end{ex}

\noindent
As a direct consequence of Theorem \ref{thm:conjugacy}, it is very easy to determine whether or not two graphically cyclically reduced words which do not contain floating syllables are conjugate, since it is sufficient to know if one of these words can be obtained from the other by permuting consecutive syllables which belong to adjacent vertex-groups and by performing cyclic permutations. Because this operation only requires to be able to solve the word problem in vertex-groups, it follows that this partial conjugacy problem can be solved even if the conjugacy problem is not solvable in some vertex-groups. More precisely:

\begin{cor}
Let $\Gamma$ be a simplicial graph and $\mathcal{G}$ a collection of groups indexed by $V(\Gamma)$. Assume that the word problem is solvable in every vertex-group. There exists an algorithm deciding whether or not two graphically reduced words which do not contain floating syllables are conjugate.
\end{cor}\qed

\section{Conjugation length functions}\label{section:CLF}

\noindent
In this section, our goal is to estimate the conjugacy length function of a graph product. First, recall that:

\begin{definition}
Let $G$ be a group and $S \subset G$ a finite generating set. The \emph{conjugacy length function} of $G$ (with respect to $S$) is
$$\mathrm{CLF}_G : n \mapsto \max\limits_{a,b \in G, \ \|a\|_S + \|b \|_S \leq n} \ \min\limits_{c \in G, \ cac^{-1}=b} \|c \|_S$$
\end{definition}

\noindent
Before stating and proving the main result of this section, the following definition is needed:

\begin{definition}
Let $\Gamma$ be a simplicial graph, $\mathcal{G}$ a collection of groups indexed by $V(\Gamma)$ and $a \in \Gamma \mathcal{G}$ an element. The \emph{floating set} of $a$ is the set of vertices of $\Gamma$ labelled by the vertex-groups which contain the floating syllables of any graphically cyclically reduced word in the conjugacy class of $a$. 
\end{definition}

\noindent
Notice that this definition makes sense according to Theorem \ref{thm:conjugacy}, ie., the floating set of an element does not depend on the graphically cyclically reduced word we choose. 

\medskip \noindent
If $\Gamma$ is a finite simplicial graph and $\mathcal{G}$ a collection of finitely generated groups indexed by $V(\Gamma)$, we endow $\Gamma \mathcal{G}$ with a finite generating set of the form $\bigcup\limits_{u \in V(\Gamma)} X_u$ where $X_u$ is an arbitrary finite generating set of $G_u$ for every $u \in V(\Gamma)$. We denote by $\| \cdot \|$ the corresponding word length and by $\mathrm{CLF}_{\Gamma \mathcal{G}}$ the corresponding conjugacy length function. The main result of this section is:

\begin{thm}\label{thm:CLF}
Let $\Gamma$ be a finite graph and $\mathcal{G}$ a collection of finitely generated groups indexed by $V(\Gamma)$. If two elements $a,b\in \Gamma \mathcal{G}$ are conjugate, then there exists an element $c \in \Gamma \mathcal{G}$ satisfying $cac^{-1}=b$ such that
$$\|c \| \leq (D+1) \cdot \left( \| a \| + \|b \| \right) + \sum\limits_{u \in F} \mathrm{CLF}_{G_u}(\|a\|+ \| b \|)$$
where $D$ is the maximal diameter of a connected component of $\Gamma^{\mathrm{opp}}$ and where $F$ denotes the floating set of $a$ (or equivalently, of $b$). 
\end{thm}

\noindent
In this statement, the \emph{opposite graph} $\Gamma^{\mathrm{opp}}$ is the graph whose vertex-set is $V(\Gamma)$ and whose edges link two vertices if they are not adjacent in $\Gamma$. We begin by proving a quantified version of Lemma \ref{lem:graphicallycyclicallyreduced}.

\begin{prop}\label{prop:CLF}
For every $a \in \Gamma \mathcal{G}$, there exist $b,g \in \Gamma \mathcal{G}$ such that $a=gbg^{-1}$, such that $b$ is graphically cyclically reduced, and such that $\|b\| \leq \|a\|$ and $ \|g\| \leq \|a\|/2$. 
\end{prop}

\begin{proof}
First, write $a$ as a graphically reduced word 
$$g_1 \cdots g_p \cdot \underset{=:a'}{\underbrace{c_1 \cdots c_q}} \cdot g_p^{-1} \cdots g_1^{-1},$$ 
where $a'$ is \emph{weakly graphically cyclically reduced}, ie., it is graphically reduced and, if there exist $1 \leq i<j \leq q$ such that $c_i$ and $c_j$ belong to a common vertex-group which is adjacent to the vertex-group which contains $c_k$ for every $1 \leq k <i$ and to the vertex-group which contains $c_k$ for every $j<k \leq q$, then $c_i \neq c_j^{-1}$. $a' = c_1 \cdots c_q$. Notice that, as a consequence of Claim \ref{claim:length}, we have
$$\|a\| = \|a'\|+ 2 \sum\limits_{i=1}^p \|g_i\|.$$
Next, write $a'$ as a graphically reduced word $a_1 \cdots a_s \cdot x_1 \cdots x_r \cdot b_s \cdots b_1$ such that:
\begin{itemize}
	\item for every $1 \leq i \leq k$, the syllables $a_i$ and $b_i$ belong to the same vertex-group, say $G_{u_i}$ for some $u_i \in V(\Gamma)$;
	\item the vertices $u_1, \ldots, u_k$ are pairwise adjacent in $\Gamma$;
	\item there do not exist $1 \leq i < j \leq r$ such that $r_i$ and $r_j$ belong to a common vertex-group which is adjacent to $G_{u_k}$ for every $1 \leq k \leq s$, to the vertex-group containing $x_k$ for every $1 \leq k < i$, and to the vertex group containing $x_k$ for every $j<k \leq r$.
\end{itemize}
As a consequence of Claim \ref{claim:length}, we have
$$\|a'\| = \sum\limits_{i=1}^r \|x_i\| + \sum\limits_{i=1}^s \left( \|a_i\| + \|b_i\| \right).$$
Notice that either $\sum\limits_{i=1}^s \|a_i\| \leq \frac{1}{2} \|a'\|$ or $\sum\limits_{i=1}^s \|b_i\| \leq \frac{1}{2} \|a'\|$. In the former case, we consider
$$\begin{array}{lcl} a' & = & (a_1 \cdots a_s) \cdot \left[ x_1 \cdots x_r (b_s \cdots b_1) (a_1 \cdots a_s) \right] \cdot (a_1 \cdots a_s)^{-1} \\ \\ & = & (a_1 \cdots a_s) \cdot \underset{=:b}{\underbrace{\left[ x_1 \cdots x_r (b_sa_s) \cdots (b_1a_1)\right]}} \cdot (a_1 \cdots a_s)^{-1}. \end{array}$$
and in the latter case,
$$\begin{array}{lcl} a' & = & (b_s \cdots b_1)^{-1} \left[(b_s \cdots b_1) (a_1 \cdots a_s) x_1 \cdots x_r \right] (b_s \cdots b_1) \\ \\ & = &  (b_s \cdots b_1)^{-1} \cdot \underset{=:b}{\underbrace{\left[(b_1 a_1) \cdots (b_s a_s) x_1 \cdots x_r \right]}} \cdot (b_s \cdots b_1). \end{array}$$
From now on, we suppose that we are in the former case, the latter case being similar. 

\medskip \noindent
We claim that $b$ is graphically cyclically reduced. First, notice that all the syllables of $b$ must be non-trivial since the fact that $a'$ is weakly graphically cyclically reduced implies that $a_i \neq b_i^{-1}$ for every $1 \leq i \leq s$. Next, $x_1 \cdots x_r$ must be graphically reduced since it is a subword of the graphically reduced word $a_1 \cdots a_s x_1 \cdots x_r b_s \cdots b_1$; and $(b_sa_s) \cdots (b_1a_1)$ must be graphically reduced as well because the $u_i$'s are pairwise distinct. Consequently, the only possibility for $x_1 \cdots x_r (b_sa_s) \cdots (b_1a_1)$ not to be graphically reduced is that there exists some $1 \leq i \leq r$ such that $x_i$ belongs to $G_{u_j}$ for some $1 \leq j \leq k$ and shuffles to the end in $x_1 \cdots x_r$. But this is impossible since $x_1 \cdots x_r b_s \cdots b_1$ is graphically reduced. Thus, we have proved that $b$ is graphically reduced. Next, if $b$ is not graphically cyclically reduced, then only two cases may happen because the $u_i$'s are pairwise distinct:
\begin{itemize}
	\item either there exists some $1 \leq i \leq r$ such that $x_i$ belongs to $G_{u_j}$ for some $1 \leq j \leq k$ and shuffles to the beginning in $x_1 \cdots x_r$, but this is impossible since otherwise $a_1 \cdots a_s x_1 \cdots x_r$ would not be graphically reduced;
	\item or there exist $1 \leq i < j \leq r$ such that $r_i$ and $r_j$ belong to a common vertex-group which is adjacent to the $G_{u_k}$'s, to the vertex-group containing $x_k$ for every $k>j$, and to the vertex-group containing $x_k$ for every $k<i$, but this is impossible because otherwise $a_1 \cdots a_s x_1 \cdots x_r b_s \cdots b_1$ would not be graphically reduced.
\end{itemize}
Thus, we have proved our claim, ie., the word $b$ is indeed graphically cyclically reduced. Now, because
$$a =  \underset{=:g}{\underbrace{(g_1 \cdots g_p a_1 \cdots a_s)}} \cdot b \cdot (g_1 \cdots g_p a_1 \cdots a_s)^{-1}$$
it only remains to estimate the lengths of $b$ and $g$ in order to conclude the proof of our proposition. Using Claim \ref{claim:length}, we find that
$$\|b \| = \sum\limits_{i=1}^r \|x_i\| + \sum\limits_{i=1}^s \|b_ia_i \| \leq \sum\limits_{i=1}^r \|x_i\| + \sum\limits_{i=1}^s \left( \|b_i \| + \|a_i \| \right)= \|a'\| \leq \|a\|$$
and
$$\| g \| \leq \sum\limits_{i=1}^p \|g_i\| + \sum\limits_{i=1}^s \|a_i\| \leq \sum\limits_{i=1}^p \|g_i\| + \frac{1}{2} \|a'\| = \frac{1}{2}\|a\|,$$
which are the desired inequalities.
\end{proof}

\begin{proof}[Proof of Theorem \ref{thm:CLF}.]
Let $a,b \in \Gamma \mathcal{G}$ be two conjugate elements. According to Proposition \ref{prop:CLF}, there exist $a',b',g,h \in \Gamma \mathcal{G}$ such that $a'$ and $b'$ are graphically cyclically reduced, such that $a=ga'g^{-1}$ and $b=hb'h^{-1}$, and such that $\left\{ \begin{array}{l} \|a'\| \leq \|a\|, \ \|g \| \leq \|a\|/2 \\ \|b'\| \leq \|b\|, \ \|h \| \leq \|b\|/2 \end{array} \right.$. Because $a'$ and $b'$ are conjugate, we know from Theorem \ref{thm:conjugacy} that there exist graphically reduced words $x_1 \cdots x_r p_1 \cdots p_n$ and $y_1 \cdots y_r q_1 \cdots q_n$ such that
\begin{itemize}
	\item $a'=x_1 \cdots x_r p_1 \cdots p_n$ and $b'=y_1 \cdots y_r q_1 \cdots q_n$;
	\item the $p_i$'s and the $q_i$'s are floating syllables of $a$ and $b$ respectively, and $p_i$ and $q_{i}$ are conjugate in a vertex-group;
	\item $y_1 \cdots y_r$ can be obtained from $x_1 \cdots x_r$ by permuting two consecutive syllables which belong to adjacent vertex-groups and by performing cyclic permutations. 
\end{itemize} 
Let $c_1, \ldots, c_n$ be shortest conjugators such that $c_i p_ic_i^{-1}=q_{i}$ for every $1 \leq i \leq n$. Also, we know from the main theorem of \cite{Duboc} (see page 173 for the estimate of the constant) that there exists a sequence of words $m_1, \ldots, m_s$ such that $s \leq D$, $m_1=x_1 \cdots x_r$ and $m_s= y_1 \cdots y_r$, and such that $m_i$ is a cyclic permutation of $m_{i+1}$ for every $1 \leq i \leq s-1$. As a consequence, there exists a word $d$ written over $\{x_1, \ldots, x_r \}$ such that $dx_1 \cdots x_r d^{-1}= y_1 \cdots y_r$ and $\| d \| \leq D \cdot \|x_1 \cdots x_r \|$. Notice that
$$\begin{array}{lcl} \|d\| & \leq & \displaystyle D \cdot \|x_1 \cdots x_r \| = \frac{D}{2} \cdot \left( \|x_1 \cdots x_r\| + \|y_1 \cdots y_r\| \right) \\ \\ & \leq & \displaystyle \frac{D}{2} \cdot \left( \|a'\| + \|b'\| \right) \leq \frac{D}{2} \cdot \left( \|a\| + \|b\| \right). \end{array}$$
Moreover, $d$ necessarily commutes with $c_1, \ldots, c_n$. 

\medskip \noindent
We claim that $c:= hdc_1 \cdots c_ng^{-1}$ is the element we are looking for. First, we have
$$\begin{array}{lcl} cac^{-1} & = & (hdc_1 \cdots c_n) \cdot g^{-1}ag \cdot (hdc_1 \cdots c_n)^{-1}= (hdc_1 \cdots c_n) \cdot a' \cdot (hdc_1 \cdots c_n)^{-1} \\ \\ & = & (hdc_1 \cdots c_n) \cdot x_1 \cdots x_r p_1 \cdots p_n \cdot (hdc_1 \cdots c_n)^{-1} \\ \\ & = & h \cdot \left( dx_1 \cdots x_rd^{-1} \cdot c_1p_1c_1^{-1} \cdots c_n p_n c_n^{-1} \right) \cdot h^{-1} \\ \\ & = & h \cdot \left( y_1 \cdots y_r \cdot q_{1} \cdots q_{n} \right)\cdot h^{-1}= h \cdot b' \cdot h^{-1}=b\end{array}$$
Next, we have
$$\begin{array}{lcl} \|c\| & \leq & \displaystyle \|h \| + \| g \| + \|d\| + \sum\limits_{i=1}^n \|c_i\| \\ \\ & \leq & \displaystyle \frac{\|a\|}{2} + \frac{\|b\|}{2} + D \cdot \left( \|a\|+ \|b\| \right) + \sum\limits_{i=1}^n \|c_i \| \\ \\ & \leq & \displaystyle (D+1) \cdot (\| a\| + \|b\|) + \sum\limits_{u \in F} \mathrm{CLF}_{G_u}(\|a\|+\|b\|) \end{array}$$
where the last inequality is justified by: 
$$\|p_i\| + \|q_i\| \leq \| a'\| + \|b'\|  \leq \|a\| + \|b\|.$$
The proof of our theorem is complete. 
\end{proof}

\noindent
As a first application of Theorem \ref{thm:CLF}, we are able to estimate efficiently the conjugacy length function of a graph product. Let us begin with a particular case, which includes right-angled Artin groups and right-angled Coxeter groups.

\begin{cor}\label{cor:CLFraag}
Let $\Gamma$ be a finite simplicial graph and $\mathcal{G}$ a collection of cyclic groups indexed by $V(\Gamma)$. Then
$$\mathrm{CLF}_{\Gamma \mathcal{G}} (n) \leq (D+1) \cdot n \leq \left(\# V(\Gamma)+1 \right) \cdot n$$
for every $n \geq 1$, where $D$ denotes the maximal diameter of a component of $\Gamma^{\mathrm{opp}}$. 
\end{cor}

\begin{proof}
Because the conjugacy length function of an abelian is zero, the desired conclusion follows immediately from Theorem \ref{thm:CLF}. 
\end{proof}

\noindent
The general case is considered by our next corollary.

\begin{cor}
Let $\Gamma$ be a finite simplicial graph and $\mathcal{G}$ a collection of finitely generated groups indexed by $V(\Gamma)$. Then
$$\max\limits_{u \in V(\Gamma)} \mathrm{CLF}_{G_u}(n) \leq \mathrm{CLF}_{\Gamma \mathcal{G}}(n) \leq (D+1) \cdot n + \max\limits_{\Delta \subset \Gamma \ \text{complete}} \sum\limits_{u \in V(\Delta)} \mathrm{CLF}_{G_u}(n)$$
for every $n \geq 1$, where $D$ denotes the maximal diameter of a connected component of the opposite graph $\Gamma^{\mathrm{opp}}$. 
\end{cor}

\begin{proof}
The upper bound is an immediate consequence of Theorem \ref{thm:CLF}. In order to show the lower bound, it is sufficient to prove the following general but elementary statement about retract subgroups:

\begin{claim}
Let $G$ be a group endowed with a finite generating set $X$. Fix a subgroup $H \subset G$ and assume that there exists a surjective morphism $r : G \twoheadrightarrow H$ such that $r(h)=h$ for every $h \in H$. We endow $H$ with the finite generating set $r(X)$. Then
$$\mathrm{CLF}_H(n) \leq \mathrm{CLF}_G(K \cdot n)$$
for every $n \geq 1$, where $K= \max\limits_{s \in X} \| r(s) \|_G$.
\end{claim}

\noindent
Fix some $n \geq 1$ and let $a,b \in H$ be two conjugate elements satisfying $\|a\|_H + \|b\|_H \leq n$. It is not difficult to show that 
$$\frac{1}{K} \cdot \|h\|_G \leq \|h \|_H \leq \|h\|_G$$
for every $h \in H$, so we have $\|a\|_G + \|b\|_G \leq K \cdot n$. As a consequence, there exists some $c \in G$ such that $a=cbc^{-1}$ and such that $\|c\|_G \leq \mathrm{CLF}_G(K \cdot n)$. Notice that
$$a=r(a)=r \left( cbc^{-1} \right) = r(c)r(b) r(c^{-1}) = r(c) b r(c^{-1})$$
and that $\|r(c)\|_H \leq \| r(c) \|_G \leq \mathrm{CLF}_G(K \cdot n)$. Thus, we have proved that, for every conjugate elements $a,b \in H$ satisfying $\|a\|_H+ \|b\|_H \leq n$, there exists a conjugator in $H$ which has length at most $\mathrm{CLF}_G(K \cdot n)$ in $H$. We conclude that $\mathrm{CLF}_H(n) \leq \mathrm{CLF}_G(K \cdot n)$, as desired.  
\end{proof}

\medskip \noindent
Interestingly, it follows from Theorem \ref{thm:CLF} that it may be possible to find conjugators of linear lengths between specific elements of a graph product, without any assumption on the vertex-groups. For instance, an immediate consequence of Theorem \ref{thm:CLF} is:

\begin{cor}
Let $\Gamma$ be a finite graph and $\mathcal{G}$ a collection of finitely generated groups indexed by $V(\Gamma)$. Fix two conjugate elements $a,b\in \Gamma \mathcal{G}$ and assume that $a$ and $b$ do not contain floating syllables. Then there exists an element $h \in \Gamma \mathcal{G}$ satisfying $hah^{-1}=b$ such that
$$\|h \| \leq (D+1) \cdot \left( \| a \| + \|b \| \right),$$
where $D$ denotes the maximal diameter of a component of $\Gamma^{\mathrm{opp}}$. 
\end{cor}\qed

\noindent
Another family of elements for which we get a linear control on the size of conjugators is given by the next corollary. 

\begin{cor}
Let $\Gamma$ be a finite graph and $\mathcal{G}$ a collection of finitely generated groups indexed by $V(\Gamma)$. Fix two conjugate infinite-order elements $a,b\in \Gamma \mathcal{G}$ and assume that $a$ and $b$ do not belong to conjugates of vertex-groups and that they have virtually cyclic centralisers. Then there exists an element $h \in \Gamma \mathcal{G}$ satisfying $hah^{-1}=b$ such that
$$\|h \| \leq (D+1) \cdot \left( \| a \| + \| b \| \right) + \mathrm{clique}(\Gamma) \cdot \max \left\{ \mathrm{diam}(G_u) \mid |G_u|<+ \infty\right\},$$
where $D$ denotes the maximal diameter of a component of $\Gamma^{\mathrm{opp}}$ and where $\mathrm{clique}(\Gamma)$ denotes the maximal cardinality of a complete subgraph of $\Gamma$. 
\end{cor}

\begin{proof}
Because $a$ does not belong to a conjugate of a vertex-group and because it is an infinite-order element with a virtually cyclic centraliser, it follows that $G_u$ is finite for every $u \in F$, where $F$ denotes the floating set of $a$. Consequently, we have
$$\sum\limits_{u \in F} \mathrm{CLF}_{G_u}( \|a\|+ \|b\|) \leq \sum\limits_{u \in F} \mathrm{diam}(G_u) \leq \mathrm{clique}(\Gamma) \cdot \max \left\{ \mathrm{diam}(G_u) \mid |G_u|<+ \infty\right\}.$$
The desired conclusion now follows directly from Theorem \ref{thm:CLF}. 
\end{proof}

\addcontentsline{toc}{section}{References}

\bibliographystyle{alpha}
{\footnotesize\bibliography{GPvanKampen}}

\end{document}